\documentclass[amstex,11pt,reqno,english]{amsart}
\usepackage{amsmath,amsfonts,amssymb,amsthm,enumerate,hyperref,multicol,graphicx}
\usepackage{sidecap}
\usepackage{lineno,hyperref}
\usepackage{mathrsfs,  amssymb, bbm}
\usepackage{wrapfig}
\usepackage{babel,blindtext}
\usepackage{caption, float}
\usepackage{subfig}
\usepackage{amsmath}
\newtheorem{lemma}{Lemma}
\newtheorem{thm}{Theorem}

\newtheorem{corollary}{Corollary}
\newtheorem{remark}{Remark}
\numberwithin{equation}{section}
\numberwithin{thm}{section}
\numberwithin{corollary}{section}
\numberwithin{remark}{section}
\numberwithin{lemma}{section}

\usepackage{mathtools}

\newtheorem*{ack}{Acknowledgement}

\allowdisplaybreaks
\begin{document}
\title[Tamed Milstein Scheme for SDEwMS]{On Explicit Tamed Milstein-type scheme for Stochastic Differential Equation with Markovian Switching}

\author{Chaman Kumar}
\address{Department of Mathematics\\
Indian Institute of Technology Roorkee, India}
\email{C.Kumarfma@iitr.ac.in}

\author{Tejinder Kumar}
\address{Department of Mathematics\\
Indian Institute of Technology Roorkee, India}
\email{tejinder.dma2017@iitr.ac.in}

\maketitle

\begin{abstract}
We propose a  new tamed Milstein-type scheme for stochastic differential equation with Markovian switching when drift coefficient is assumed to grow super-linearly. The strong rate of convergence is shown to be equal to $1.0$ under mild regularity (e.g. once differentiability)  requirements on  drift and diffusion coefficients. Novel techniques are developed to tackle two-fold difficulties arising due to jumps of the Markov chain and the reduction of regularity requirements on the coefficients. 
\end{abstract}

\section{\bf{Introduction}}
\label{intro}
The stochastic differential equation with Markovian switching (SDEwMS) have found several important applications in real life situations, see for example, \cite{bao2016, mao2006,  nguyen2012, sethi1994, zhang1998, zhang2001} and references therein.  Often, explicit solutions of such equations are not known and hence it becomes necessary to find their approximate solutions. After the seminal work of \cite{giles2008},  strong convergence of numerical approximation of stochastic differential equations (SDEs) have gained a great deal of importance due to its requirement for multi-level Monte Carlo scheme. The strong convergence of  easily implementable and computationally efficient Euler-Maruyama (EM) scheme of SDEwMS is well understood in the literature when drift and diffusion coefficients satisfy global Lipschitz conditions, see for example \cite{mao2006, nguyen2012} and references therein. A new variant of Euler-type schemes, often known as tamed EM schemes,  are recently developed to obtain approximate solutions of SDEwMS when drift/diffusion coefficient(s) is/are allowed grow super-linearly, see for example  \cite{dareiotis2016, kumar2017, nguyen2018}. The rates of convergence of   EM and tamed EM schemes for SDEwMS are known to be equal to  $1/2$.  However,  very scant literature is available on higher order numerical approximations of SDEwMS even for the simple case where both drift and diffision coefficients satisfy global Lipschitz conditions. The main reason for such a scarse literature on higher order approximations of SDEwMS is the absence of It\^o's-Taylor expansion for coefficients which depend on Markov chain such as those encountered in SDEwMS. It can be remarked that  higher order approximations of stochastic differential equations (SDEs) are well developed due to the presence of It\^o's-Taylor expansion, see  \cite{kloeden1992} for detailed discussion. Furthermore,  different variants of tamed Milstein scheme for SDE have been studied by  \cite{beyn2017, kumar2017, kumar2019, tretyakov2013} where authors estalish that the strong convergence of their variant of tamed Milstein schemes for SDE is $1.0$ when either drift or diffusion or both the coefficients satisfy non-global Lipschitz condition and can grow super-linearly. Recently, Milstein-type schemes of SDEwMS is investigated  in \cite{nguyen2017} where authors establish that the rate of strong convergence of their scheme is equal to $1.0$ when both drift and diffusion coefficients satisfy global Lipschitz conditions. In order to show their results, they derive a new version of It\^o's formula for coefficient  which depends on the Markov chain (such as drift and diffusion coefficients of SDEwMS). Moreover, they assume that both drift and diffusion coefficients are twice differentiable and satisfy some additional global Lipschitz conditions. In this article, we propose a tamed Milstein scheme for SDEwMS and investigate its strong convergence under more relaxed assumptions than those made in  \cite{nguyen2017}. More precisely,  strong rate of convergence of our scheme is shown to be equal to $1.0$,  under much improved conditions than those made in \cite{nguyen2017} such as, 
\begin{enumerate}[(a)]
\item drift coefficient satisfies non-global Lipschitz condition and can grow super-linearly,
\item both drift and diffusion coefficients are assumed to be only once differentiable, 
\item drift and its first order partial derivatives  satisfy polynomial Lipschitz condition,  and
\item diffusion and its first order partial derivatives  satisfy global Lipschitz condition.   
\end{enumerate}
These relaxed assumptions allow us to cover more SDEwMS than those covered  by  \cite{nguyen2017}. Moreover, new techniques are developed to tackle challenges posed - (i) by these relaxed assumptions as mentioned above,  and - (ii)  by  jumps arising due to Markovian switching. The approach developed in this paper can be used to study strong convergence of other higher order schemes and to reduce the regularity requirements on the coefficients of SDEwMS.   To the best of our knowledge, this is the first article which deals with tamed Milstein-type scheme for SDEwMS.

We conclude this section by introducing some notations used throughout in this  article. For a vector $b\in \mathbb{R}^d$ and a matrix $\sigma \in \mathbb{R}^{d\times m}$, same notation $|\cdot|$ is used for both  Euclidean and Hilbert-schmidt norms, which can be understood from each instance in this article.    Also, $\sigma^{(l)}\in\mathbb{R}^d$ stands for  $l$-th column of a matrix $\sigma\in\mathbb{R}^{d\times m}$.  For vectors $x,y\in\mathbb{R}^d$, $xy$ denotes  their inner product. The matrix transpose is abbreviated by $\sigma^T$ for $\sigma\in \mathbb{R}^{d \times m}$ . Furthermore, for a function $f:\mathbb{R}^d \to \mathbb{R}^d$, $\mathcal{D}f$ returns a $d\times d$ matrix with  $\frac{\partial f^i(\cdot)}{\partial x^j}$ as its $(i,j)$-th entry for  any $i,j=1,\ldots,d$. The indicator function of set $A$ is denoted by $\mathbb{I}\{A\}$. $\lfloor a \rfloor$ stands for the integer part of a positive real number $a$. Also, $a\wedge b$ $=$ $\min (a, b)$ for real positive numbers $a$ and $b$. Finally, symbol $K$ is used throughout this article for a generic positive constant $K$  which vary from place to place.

\section{\bf{Main Result}}
Suppose that  $ \big( \Omega ,   \mathscr{F}, P \big)$  is  a complete probability space. Let $W:=\{W_t; t\geq 0\}=\{(W_t^l)_{l=1}^m; t\geq 0\}$ be an $\mathbb{R}^m$-valued standard Wiener process. For a fixed positive integer $m_0$, let $\alpha:=\{\alpha_t;t\geq 0\}$ be a continuous-time Markov chain defined on a finite state space $\mathcal{S}=\{1,2,\ldots,m_0 \}$ with generator  $Q=(q_{i_0j_0}; i_0, j_0 \in \mathcal{S})$ where, 
\begin{align} 
P(\alpha_{t+\delta}=j_0|\alpha_t=i_0)=
\begin{cases}
q_{i_0j_0}\delta+o(\delta), &\text{if } i_0\neq j_0,\\
1+q_{i_0j_0}\delta+o(\delta), &\text{if } i_0= j_0,
\end{cases} \notag
\end{align}
 $q_{i_0j_0}\geq 0$,  for any $i_0\neq j_0\in \mathcal{S}$ and $q_{i_0i_0}=-\sum_{j_0\neq i_0} q_{i_0j_0}$ for any $i_0\in\mathcal{S}$. Also, assume that $b:\mathbb{R}^d\times \mathcal{S} \to \mathbb{R}^d$ and $\sigma:\mathbb{R}^d\times \mathcal{S} \to \mathbb{R}^{d \times m}$  are measurable functions. In this article, the following $d$-dimensional stochastic differential equation with Markovian switching (SDEwMS) is considered, for any $t\in[0,T]$, almost surely, 
\begin{align} \label{eq:sdems}
X_t=X_0+\int^t_0b(X_s,\alpha_s)ds+\sum_{l=1}^m\int^t_0 \sigma^{(l)}(X_s,\alpha_s)dW_s^l
\end{align} 
where  initial value $X_0$ is an $\mathscr{F}_0$-measurable random variable taking values in $\mathbb{R}^d$. Also, $X_0$,  $W$ and $\alpha$ are assumed to be independent. Further, suppose that $\mathscr{F}^W$ and $\mathscr{F}^\alpha$ are  smallest filtrations generated by $(X_0,W)$  and $\alpha$. Define $\mathscr{F}_t:=\mathscr{F}_t^W \vee \mathscr{F}_t^\alpha$ for any $t\geq 0$. 

We now proceed to introduce our Milstein-type scheme for SDEwMS \eqref{eq:sdems}. For this, let us partition the interval $[0,T]$ into sub-intervals each of length $h=1/n$ for any $n\in\mathbb{N}$. Set $t_k:=k/n$, $\Delta_k t:=t_{k+1}-t_k$ and $\Delta_kW:=W_{t_{k+1}}-W_{t_{k}}$ for $k=0,\ldots,nT-1$. Further, let $N_{(t_k,t_{k+1})}$ be the number of jumps of the chain $\alpha$ and $\tau_1^k$ be the first first jump-time in interval $(t_k,t_{k+1})$ for any $k=0,\ldots,nT-1$. Further, $\{b^n;n\in\mathbb{N}\}$ is a sequence of functions from $\mathbb{R}^d$ to $\mathbb{R}^d$.  Our tamed Milstein-type scheme for SDEwMS at grid point $t_{k+1}$ is given by, 
\begin{align}
X_{t_{k+1}}^n & =X^n_{t_k}+b^n(X^n_{t_k},\alpha_{t_k})\Delta_kt+\sum_{l=1}^m \sigma^{(l)}(X^n_{t_k}, \alpha_{t_k}) \Delta_kW^l \notag
\\
&+\sum_{l,l_1=1}^m \int_{t_{k}}^{t_{k+1}} \int_{t_{k}}^s \mathcal{D} \sigma^{(l)}(X^n_{t_k}, \alpha_{t_k})\sigma^{(l_1)}(X^n_{t_k},\alpha_{t_k})dW_u^{l_1}dW_s^l \notag
\\
&+\sum_{l=1}^m \mathbb{I}{\{N_{(t_k,t_{k+1})}=1\}} \Big(\sigma^{(l)}(X^n_{t_k},\alpha_{t_{k+1}})-\sigma^{(l)}(X^n_{t_k},\alpha_{t_k})\Big)\Big(W^l_{t_{k+1}}-W^l_{\tau_1^k}\Big) \label{eq:scheme:dis}
\end{align}
almost surely for any $k=0,\ldots,nT-1$ where initial value $X^n_0$ is an $\mathscr{F}_0$-measurable random variable in $\mathbb{R}^d$. The motivation for the above scheme \eqref{eq:scheme:dis} comes from \cite{kumar2019a} where a Milstein-type scheme is derived for SDEwMS when both drift and diffusion coefficients satisfy global Lipschitz conditions. If the following commutative conditions, 
\begin{align}
 \mathcal{D} \sigma^{(l)}(x, i_0)\sigma^{(l_1)}(x,i_0) = \mathcal{D} \sigma^{(l_1)}(x, i_0)\sigma^{(l)}(x,i_0) \label{eq:comm}
\end{align}
hold for all $x\in\mathbb{R}^d$, $i_0\in\mathcal{S}$ and $l,l_1=1,\ldots,m$,  then  scheme \eqref{eq:scheme:dis} becomes, 
\begin{align*}
X^n_{t_{k+1}}& =X^n_{t_k}+b^n(X^n_{t_k},\alpha_{t_k})\Delta_kt+\sum_{l=1}^m \sigma^{(l)}(X^n_{t_k}, \alpha_{t_k}) \Delta_kW^l \notag
\\
&+\frac{1}{2}\sum_{l,l_1=1}^m  \sigma^{(l)}(X^n_{t_k}, \alpha_{t_k})\sigma^{(l_1)}(X^n_{t_k},\alpha_{t_k}) (\Delta_k W_t^{l_1}\Delta_k W_t^l-\mathbb{I}\{l=l_1\}\Delta_k t ) \notag
\\
&+\sum_{l=1}^m \mathbb{I}{\{N_{(t_k,t_{k+1})}=1\}} \Big(\sigma^{(l)}(X^n_{t_k},\alpha_{t_{k+1}})-\sigma^{(l)}(X^n_{t_k},\alpha_{t_k})\Big)\Big(W^l_{t_{k+1}}-W^l_{\tau_1^k}\Big) 
\end{align*}
almost surely for any $k=0,\ldots,nT-1$ . The above scheme can be implemented easily on computer. However, if the commutative condition \eqref{eq:comm} does not hold, then one can refer to \cite{davie2015} to calculate  iterated Brownian integrals  appearing in the scheme \eqref{eq:scheme:dis}.  We now proceed towards the construction of continuous version of  the scheme \eqref{eq:scheme:dis}. For this, we introduce a martingale associated with Markov chain $\alpha$ by adopting the approach of  \cite{nguyen2017}.  For each $i_0, j_0 \in \mathcal{S}$ and $i_0\neq j_0$, define, 
\begin{align}
[M_{i_0j_0}](t):=\sum_{0\leq s\leq t}\mathbb{I}\{\alpha_{s-}=i_0\} \mathbb{I} &\{\alpha_s=j_0\}, \,\,\, \langle M_{i_0j_0}\rangle(t):=\int_0^t q_{i_0j_0}\mathbb{I}\{\alpha_{s-}=i_0\} ds, \notag
\\
M_{i_0j_0}(t)&:= [M_{i_0 j_0}](t)-\langle M_{i_0j_0}\rangle(t)\notag 
\end{align}
almost surely for any $t\in[0,T]$. The stochastic processes $\{[M_{i_0 j_0}](t); t\in[0,T]\}$ and $\{\langle M_{i_0j_0}\rangle(t); t\in[0,T]\}$ are respectively optional and predictable quadratic variations. Also, the stochastic process  $\{M_{i_0j_0}(t); \, t\in[0,T]\}$ is a purely discontinuous process and is a square integrable martingale with respect to filtration $\{\mathscr{F}_{t}^\alpha; \, t \in [0,T]\}$ with $M_{i_0j_0}(0)=0$ almost surely. For the convenience of notation, set $M_{i_0i_0}(t)=0$ for any $i_0\in\mathcal{S}$ and $t\in[0,T]$. Further, $\kappa(n,t):=\lfloor nt \rfloor/n$ for any $n\in\mathbb{N}$. The continuous version of scheme \eqref{eq:scheme:dis} is given by 
\begin{align}\label{eq:scheme}
X_t^n=X^n_0+\int^t_0b^n(X^n_{\kappa(n,s)},\alpha_{\kappa(n,s)})ds+\sum^m_{l=1}\int^t_0\tilde{\sigma}^{(l)}(s,X^n_{\kappa(n,s)},\alpha_{\kappa(n,s)})dW^l_s
\end{align}
almost surely for any $t\in[0,T]$ and $n\in\mathbb{N}$. In the above,  $\tilde{\sigma}^{(l)}\in\mathbb{R}^d$ is $l$-th column of $d\times d$ matrix $\tilde{\sigma}$ and is given by, 
\begin{align*}
\tilde{\sigma}(s,X^n_{\kappa(n,s)},\alpha_{\kappa(n,s)})&:=\sigma(X^n_{\kappa(n,s)},\alpha_{\kappa(n,s)})+\sigma_1(s,X^n_{\kappa(n,s)},\alpha_{\kappa(n,s)})+\sigma_2(s,X^n_{\kappa(n,s)},\alpha_{\kappa(n,s)})
\end{align*}
where  $\sigma_1$ and $\sigma_2$ in the above expression are $d\times d$ matrix with $l$-th column as, 
\begin{align*}
\sigma_1^{(l)}(s,X^n_{\kappa(n,s)},\alpha_{\kappa(n,s)})&:=\sum^m_{l_1=1}\int^s_{\kappa(n,s)}\mathcal{D}\sigma^{(l)}(X^n_{\kappa(n,r)},\alpha_{\kappa(n,r)})\sigma^{(l_1)}(X^n_{\kappa(n,r)},\alpha_{\kappa(n,r)})dW^{l_1}_r,
\\
\sigma_2^{(l)}(s,X^n_{\kappa(n,s)},\alpha_{\kappa(n,s)})&:=\sum_{i_0\neq j_0}\int^s_{\kappa(n,s)}\mathbb{I}\{N_{(\kappa(n,s),s)}=1\} (\sigma^{(l)}(X^n_{\kappa(n,r)},j_0)-\sigma^{(l)}(X^n_{\kappa(n,r)},i_0))d[M_{i_0j_0}](r),
\end{align*}
almost surely for any $s\in[0,T]$, $l=1,\ldots,m$ and $n\in\mathbb{N}$. Notice that the continuous version of the Milstein-type scheme \eqref{eq:scheme} coincides with the discrete version of the scheme \eqref{eq:scheme:dis} at the points of discretization  $t_0,t_1,\ldots,t_{nT}$. 

Let $p\geq 2$, $\rho, \rho_1\geq 0$  be  fixed constants. We make the following assumptions. 
\newline  
\noindent 
\textbf{Assumption $H$-1.} There exists a  constant $L>0$ such that  $E|X_0|^{p}   < \infty$. 
\newline 
\noindent 
\textbf{Assumption $H$-2.} There exists a constant $L>0$ such that, for every $i_0\in \mathcal{S}$,
\begin{align*}
(x-y)(b(x,i_0)-b(y,i_0))\vee |\sigma(x,i_0)-\sigma(y,i_0)|^2\leq L|x-y|^2
\end{align*}
 for any  $x,y\in \mathbb{R}^d$. 
 \newline 
\noindent 
\textbf{Assumption $H$-3.} There exists  a constant $L>0$  such that, for every $i_0\in \mathcal{S}$,
\begin{align*}
|\mathcal{D}b(x,i_0)-\mathcal{D}b(y,i_0)|&\leq L(1+|x|+|y|)^{\rho-1}|x-y|
\\
|\mathcal{D}\sigma^{(l)}(x,i_0)-\mathcal{D}\sigma^{(l)}(y,i_0)| &\leq L|x-y|, \,\,  l=1,\ldots, m
\end{align*} 
for any $x,y\in \mathbb{R}^d$.
\newline  
\noindent 
\textbf{Assumption $H$-4.} There exists a  constant $L>0$ such that  $ E|X_0^n|^{p}  < \infty$ and $E|X_0-X_0^n|^2\leq Ln^{-2}$, for every $n\in\mathbb{N}$.
 \newline 
\noindent 
\textbf{Assumption $H$-5.} There exists a constant  $L>0$ such that 
\begin{align*}
xb^n(x,i_0)\leq L(1+|x|)^2, \quad |b^n(x,i_0)|\leq L \min\{n^{\frac{1}{2}}(1+|x|),(1+|x|)^{\rho_1+1}\}
\end{align*}
for any $x\in\mathbb{R}^d$, $n\in\mathbb{N}$ and $i_0\in\mathcal{S}$. 
  \newline 
\noindent 
\textbf{Assumption $H$-6.} There exists a constant  $L>0$ such that
\begin{align*}
|b(x,i_0)-b^n(x,i_0)|\leq Ln^{-1}(1+|x|)^{\rho_1+1}
\end{align*}
for any $x\in\mathbb{R}^d$, $n\in\mathbb{N}$ and $i_0\in\mathcal{S}$.
\newline
\begin{remark}\label{rem:growth}
Due to Assumptions $H$-2 and $H$-3,  for every $i_0\in\mathcal{S}$, there exists a constant $L>0$ such that, 
\begin{align*}
xb(x,i_0)\leq L(1+|x|)^2, \quad  |b(x,i_0)|  & \leq  L(1+|x|)^{\rho+1}, \quad |\sigma(x,i_0)|\leq L(1+|x|)
\\
 |\mathcal{D}b(x,i_0)|\leq L(1+|x|)^{\rho} ,\quad &  |\mathcal{D} \sigma^{(l)}(x,i_0)| \leq L,\, l=1,\ldots, m
 \\
 |b(x,i_0)-b(y,i_0)|\leq & L(1+|x|+|y|)^{\rho}|x-y|,
\end{align*}
for any $x,y\in\mathbb{R}^d$. 
\end{remark}
The following theorem states the main result of this article. 
\begin{thm} \label{thm:main}
Let Assumptions $H$-1 to $H$-6 be satisfied. Then, the Milstein-type scheme defined in equation \eqref{eq:scheme} converges in $\mathcal{L}^2$-sense to the true solution of SDEwMS \eqref{eq:sdems} with a rate of convergence equal to $1$. In other words,  there exists a constant $K>0$ (independent of $n$) such that the following holds, 
$$
\sup_{t\in[0,T]}E|X_t-X^n_t|^2 \leq K n^{-2}
$$
for any $n\in \mathbb{N}$, where $0<h=(1/n)<1/(2q)$ with $q=\max\{-q_{i_0i_0};i_0\in\mathcal{S}\}$. 
\end{thm}

A simple example of tamed Milstein scheme for SDEwMS can be obtained by taking 
\begin{align*}
b^n(x,i_0)=\frac{b(x,i_0)}{1+n^{-1}|x|^{2\rho}}
\end{align*}
for any $n\in\mathbb{N}$, $x\in\mathbb{R}^d$ and $i_0\in\mathcal{S}$. Clearly, Assumption H-5 and H-6 are satisfied with $\rho_1=3\rho$. 
\begin{remark}
The technique developed in this paper can be extended to the case when both drift and diffusion coefficients can grow super-linearly. Moreover, higher order schemes can be also be investigated on the line similar to the one developed in this paper.  The main approach would be to look for an It\^o's-Taylor expansion (similar to one given in \cite{kloeden1992})  for drift and diffusion coefficients which depend on the Markov chain $\alpha$ as in the case of SDEwMS. The strong rate of convergence can then be shown to be equal to any desired order by appropriately adding terms from this expansion and accordingly tame the terms of such expansion. Notice that Milstein-type scheme \eqref{eq:scheme} for SDEwMS consists of an additional term in form of $\sigma_2$ when compared with the corresponding Milstein-type schemes for SDE \cite{ kloeden1992, kumar2017, kumar2019}. One would expect similar additional terms in higher order schemes of SDEwMS. Further, one can reduce regularity requirements on the coefficients by adapting the approach developed in this article while proving strong convergence of higher order schemes. 
\end{remark}

\section{\bf{Moment Bound.}}
This section is devoted to proving the moment bounds of  SDEwMS \eqref{eq:sdems} and its scheme \eqref{eq:scheme}. The proof of the moment bounds of SDEwMS \eqref{eq:sdems}, which is stated in the following lemma,  can be found in \cite{mao2006}, see for example Theorems [3.3.16, 3.3.23, 3.3.24] in this reference.

\begin{lemma}\label{lem:true moment and 1/2 rate}
Let Assumptions $H$-1 and $H$-2 be satisfied. Then,  
\begin{align*}
E\bigg(\sup_{t\in [0,T]}|X_t|^{p}\Big|\mathscr{F}_T^{\alpha}\bigg)& \leq K
\end{align*}
almost surely,  where  $K>0$ is a  constant.  
\end{lemma}
For proving  moment bounds of the scheme \eqref{eq:scheme}, one requires to establish following lemmas. 
\begin{lemma}\label{lem:diffu.bound}
Let Assumptions $H$-2 to $H$-4 be satisfied. Then,
\begin{align*}
E\Big(|\sigma_1(t,X^n_{\kappa(n,t)},\alpha_{\kappa(n,t)})|^p\big|\mathscr{F}_T^{\alpha}\Big)&\leq Kn^{-\frac{p}{2}}E\big((1+|X^n_{\kappa(n,t)}|^2)^{\frac{p}{2}}\big|\mathscr{F}_T^{\alpha}\big),
\\
E\Big(|\sigma_2(t,X^n_{\kappa(n,t)},\alpha_{\kappa(n,t)})|^p\big|\mathscr{F}_T^{\alpha}\Big)&\leq  K \mathbb{I}\{N_{(\kappa(n,t),t)}=1\}E\big((1+|X^n_{\kappa(n,t)}|^2)^{\frac{p}{2}}\big|\mathscr{F}_T^{\alpha}\big)
\\
&\leq KE\big((1+|X^n_{\kappa(n,t)}|^2)^{\frac{p}{2}}\big|\mathscr{F}_T^{\alpha}\big)
\end{align*} 
almost surely for any $t\in[0,T]$ and $n\in\mathbb{N}$, where $K>0$ does not depend on $n$. 
\end{lemma}
\begin{proof}
By the application of martingale inequality and Remark \ref{rem:growth}, one obtains,
\begin{align*}
E\Big(|\sigma_1^{(l)}(t,&X^n_{\kappa(n,t)},\alpha_{\kappa(n,t)})|^p\big|\mathscr{F}_T^{\alpha}\Big)=E\Big(\Big|\sum^m_{l_1=1}\int^t_{\kappa(n,t)}\mathcal{D}\sigma^{(l)}(X^n_{\kappa(n,s)},\alpha_{\kappa(n,s)})
\\
&\quad\quad\times\sigma^{(l_1)}(X^n_{\kappa(n,s)},\alpha_{\kappa(n,s)})dW^{l_1}_s\Big|^p\big|\mathscr{F}_T^{\alpha}\Big)
\\
\leq&Kn^{-\frac{p-2}{2}} \sum^m_{l_1=1}\int^t_{\kappa(n,t)}E\Big(|\mathcal{D}\sigma^{(l)}(X^n_{\kappa(n,s)},\alpha_{\kappa(n,s)})\sigma^{(l_1)}(X^n_{\kappa(n,s)},\alpha_{\kappa(n,s)})|^p\big|\mathscr{F}_T^{\alpha}\Big)ds
\\
\leq& Kn^{-\frac{p}{2}}E\big((1+|X^n_{\kappa(n,t)}|^2)^{\frac{p}{2}}\big|\mathscr{F}_T^{\alpha}\big)
\end{align*}
almost surely for any $t\in[0,T]$, $l=1,\ldots,m$ and $n\in\mathbb{N}$. Furthermore,  Remark \ref{rem:growth} gives,
\begin{align*}
E\Big(|\sigma_2&(t,X^n_{\kappa(n,t)},\alpha_{\kappa(n,t)})|^p\big|\mathscr{F}_T^{\alpha}\Big)=E\Big(\big| \sum_{i_0\neq j_0}\int^t_{\kappa(n,t)} \mathbb{I}\{N_{(\kappa(n,t),t)}=1\}(\sigma(X^n_{\kappa(n,s)},j_0)
\\
&\quad\quad-\sigma(X^n_{\kappa(n,s)},i_0))d[M_{i_0j_0}](s) \big|^p\big|\mathscr{F}_T^{\alpha}\Big)
\\
& \leq  E\Big(\Big|\sum_{i_0\neq j_0}\int^t_{\kappa(n,t)}\mathbb{I}\{N_{(\kappa(n,t),t)}=1\}( 1+|X^n_{\kappa(n,s)}|) d[M_{i_0j_0}](s)\Big|^p\bigg|\mathscr{F}_T^{\alpha}\Big)
\\
& \leq  E\Big(\Big|\sum_{i_0\neq j_0}\mathbb{I}\{N_{(\kappa(n,t),t)}=1\}( 1+|X^n_{\kappa(n,t)}|) ([M_{i_0j_0}](t)-[M_{i_0j_0}](\kappa(n,t))) \Big|^p\bigg|\mathscr{F}_T^{\alpha}\Big)
\\
& \leq  E\Big(\Big|\mathbb{I}\{N_{(\kappa(n,t),t)}=1\}N_{(\kappa(n,t),t)}( 1+|X^n_{\kappa(n,t)}|)  \Big|^p\bigg|\mathscr{F}_T^{\alpha}\Big)
\\
& \leq E\big((1+|X^n_{\kappa(n,t)}|^2)^{\frac{p}{2}}\big|\mathscr{F}_T^{\alpha}\big)
\end{align*}
almost surely for any $t\in[0,T]$ and $n\in\mathbb{N}$. This completes the proof. 
\end{proof}
By using Remark \ref{rem:growth} and Lemma \ref{lem:diffu.bound}, following corollary can be made.
\begin{corollary}\label{cor:tildaSigma}
Let Assumptions $H$-2 to $H$-4 be satisfied. Then, 
\begin{align*}
E\big(|\tilde{\sigma}(t,X^n_{\kappa(n,t)},\alpha_{\kappa(n,t)})|^p\big|\mathscr{F}_T^{\alpha}\big)\leq KE\big((1+|X^n_{\kappa(n,t)}|^2)^{\frac{p}{2}}\big|\mathscr{F}_T^{\alpha}\big)
\end{align*}
almost surely for any $t\in[0,T]$ and $n\in\mathbb{N}$.
\end{corollary}
Now, we proceed to establish moment bound of our Milstein-type scheme \eqref{eq:scheme}. 
\begin{lemma}\label{lem:SchemeMoment}
Let Assumptions $H$-2 to $H$-5 be satisfied. Then, for any $p\geq 2$ and $n\in\mathbb{N}$,
\begin{align*}
E\big( \sup_{0\leq t\leq T}|X^n_t|^p\big|\mathscr{F}_T^{\alpha}\big)\leq K
\end{align*}
almost surely, where constant $K>0$ does not depend on $n$.
\end{lemma}
\begin{proof}
By It\^o's formula,
\begin{align}
(1+|X^n_t|^2)^{\frac{p}{2}}=&(1+|X_0^n|^2)^{\frac{p}{2}}+p\int^t_0(1+|X^n_s|^2)^{\frac{p-2}{2}}X^n_sb^n(X^n_{\kappa(n,s)},\alpha_{\kappa(n,s)})ds\nonumber
\\
&\quad+p\int^t_0(1+|X^n_s|^2)^{\frac{p-2}{2}}X^n_s\tilde{\sigma}(s,X^n_{\kappa(n,s)},\alpha_{\kappa(n,s)})dW_s\nonumber
\\
&\quad+\frac{p}{2}\int^t_0(1+|X^n_s|^2)^{\frac{p-2}{2}}|\tilde{\sigma}(s,X^n_{\kappa(n,s)},\alpha_{\kappa(n,s)})|^2ds\nonumber
\\
&\quad+\frac{p(p-2)}{2}\int^t_0(1+|X^n_s|^2)^{\frac{p-4}{2}}|\tilde{\sigma}^T(s,X^n_{\kappa(n,s)},\alpha_{\kappa(n,s)})X^n_s|^2ds\label{eq:X splitting}
\end{align}
almost surely for any $t\in[0,T]$ and $n\in\mathbb{N}$. By using Assumption $H$-5, 
\begin{align}
p\int^t_0(&1+|X^n_s|^2)^{\frac{p-2}{2}}X^n_sb^n(X^n_{\kappa(n,s)},\alpha_{\kappa(n,s)})ds\nonumber
\\
=&p\int^t_0(1+|X^n_s|^2)^{\frac{p-2}{2}}(X^n_s-X^n_{\kappa(n,s)})b^n(X^n_{\kappa(n,s)},\alpha_{\kappa(n,s)})ds\nonumber
\\
&\quad+p\int^t_0(1+|X^n_s|^2)^{\frac{p-2}{2}}X^n_{\kappa(n,s)}b^n(X^n_{\kappa(n,s)},\alpha_{\kappa(n,s)})ds\nonumber
\\
\leq&p\int^t_0(1+|X^n_s|^2)^{\frac{p-2}{2}}\int^s_{\kappa(n,s)}|b^n(X^n_{\kappa(n,r)},\alpha_{\kappa(n,r)})|^2drds\nonumber
\\
&\quad+p\int^t_0(1+|X^n_s|^2)^{\frac{p-2}{2}}\int^s_{\kappa(n,s)}\tilde{\sigma}(r,X^n_{\kappa(n,r)},\alpha_{\kappa(n,r)})dW_rb^n(X^n_{\kappa(n,s)},\alpha_{\kappa(n,s)})ds\nonumber
\\
&\quad+K\int^t_0(1+|X^n_s|^2)^{\frac{p-2}{2}}(1+|X^n_{\kappa(n,s)}|)^2ds\notag
\end{align}
which on substituting in equation \eqref{eq:X splitting} gives, 
\begin{align}
E\Big(& \sup_{0\leq t\leq u}(1+|X^n_t|^2)^{\frac{p}{2}}\big|\mathscr{F}_T^{\alpha}\Big)\leq E(1+|X_0^n|^2)^{\frac{p}{2}}+K\int^u_0E\Big(\sup_{0\leq r\leq s}(1+|X^n_r|^2)^{\frac{p}{2}}\big|\mathscr{F}_T^{\alpha}\Big)ds\nonumber
\\
&\quad+KE\Big(\int^u_0(1+|X^n_s|^2)^{\frac{p-2}{2}}\int^s_{\kappa(n,s)}|b^n(X^n_{\kappa(n,r)},\alpha_{\kappa(n,r)})|^2drds\big|\mathscr{F}_T^{\alpha}\Big)\nonumber
\\
&\quad+KE\Big(\int^u_0(1+|X^n_s|^2)^{\frac{p-2}{2}}\Big|\int^s_{\kappa(n,s)}\tilde{\sigma}(r,X^n_{\kappa(n,r)},\alpha_{\kappa(n,r)})dW_r\Big||b^n(X^n_{\kappa(n,s)},\alpha_{\kappa(n,s)})|ds\big|\mathscr{F}_T^{\alpha}\Big)\nonumber
\\
&\quad+pE\Big( \sup_{0\leq t\leq u}\big|\int^t_0(1+|X^n_s|^2)^{\frac{p-2}{2}}X^n_s\tilde{\sigma}(s,X^n_{\kappa(n,s)},\alpha_{\kappa(n,s)})dW_s\big| \Big|\mathscr{F}_T^{\alpha}\Big)\nonumber
\\
&\quad+KE\Big(\int^u_0(1+|X^n_s|^2)^{\frac{p-2}{2}}|\tilde{\sigma}(s,X^n_{\kappa(n,s)},\alpha_{\kappa(n,s)})|^2ds\big|\mathscr{F}_T^{\alpha}\Big)\nonumber
\\
=:&E(1+|X_0^n|^2)^{\frac{p}{2}}+K\int^u_0E\Big(\sup_{0\leq r\leq s}(1+|X^n_r|^2)^{\frac{p}{2}}\big|\mathscr{F}_T^{\alpha}\Big)ds+G_1+G_2+G_3+G_4 \label{eq:G_splitting}
\end{align}
almost surely for any $u \in[0,T]$ and $n\in\mathbb{N}$. By Assumption $H$-5, 
\begin{align}
G_1&:=KE\Big(\int^u_0(1+|X^n_s|^2)^{\frac{p-2}{2}}\int^s_{\kappa(n,s)}|b^n(X^n_{\kappa(n,r)},\alpha_{\kappa(n,r)})|^2drds\big|\mathscr{F}_T^{\alpha}\Big)\nonumber
\\
 & \leq  K\int^u_0E\Big(\sup_{0\leq r\leq s}(1+|X^n_r|^2)^{\frac{p}{2}}\big|\mathscr{F}_T^{\alpha}\Big)ds\label{eq:G_1}
\end{align}
almost surely for any $u\in[0,T]$ and $n\in\mathbb{N}$. Due to Young's inequality, martingale inequality, Assumption $H$-5 and Corollary \ref{cor:tildaSigma},  one obtains, 
\begin{align}
G_2:=&KE\Big(\int^u_0(1+|X^n_s|^2)^{\frac{p-2}{2}}\Big|\int^s_{\kappa(n,s)}\tilde{\sigma}(r,X^n_{\kappa(n,r)},\alpha_{\kappa(n,r)})dW_r\Big||b^n(X^n_{\kappa(n,s)},\alpha_{\kappa(n,s)})|ds\big|\mathscr{F}_T^{\alpha}\Big)\nonumber
\\
\leq&K\int^u_0E\Big(\sup_{0\leq r\leq s}(1+|X^n_r|^2)^{\frac{p}{2}}\big|\mathscr{F}_T^{\alpha}\Big)ds\nonumber
\\
&\quad+K\int^u_0E\Big(\Big|n^{\frac{1}{2}}\int^s_{\kappa(n,s)}\tilde{\sigma}(r,X^n_{\kappa(n,r)},\alpha_{\kappa(n,r)})dW_r\Big|^{\frac{p}{2}}|n^{-\frac{1}{2}}b^n(X^n_{\kappa(n,s)},\alpha_{\kappa(n,s)})|^{\frac{p}{2}}\big|\mathscr{F}_T^{\alpha}\Big)ds\nonumber
\\
\leq &K\int^u_0E\Big(\sup_{0\leq r\leq s}(1+|X^n_r|^2)^{\frac{p}{2}}\big|\mathscr{F}_T^{\alpha}\Big)ds+K\int^u_0E\Big(n^{-\frac{p}{2}}|b^n(X^n_{\kappa(n,s)},\alpha_{\kappa(n,s)})|^p\big|\mathscr{F}_T^{\alpha}\Big)ds \nonumber
\\
&\quad+K\int^u_0E\Big(\Big|n^{\frac{1}{2}}\int^s_{\kappa(n,s)}\tilde{\sigma}(r,X^n_{\kappa(n,r)},\alpha_{\kappa(n,r)})dW_r\Big|^p\big|\mathscr{F}_T^{\alpha}\Big)ds\nonumber
\\
\leq &K\int^u_0E\Big(\sup_{0\leq r\leq s}(1+|X^n_r|^2)^{\frac{p}{2}}\big|\mathscr{F}_T^{\alpha}\Big)ds+K\int^u_0n\int^s_{\kappa(n,s)}E\Big(|\tilde{\sigma}(r,X^n_{\kappa(n,r)},\alpha_{\kappa(n,r)})|^p\big|\mathscr{F}_T^{\alpha}\Big)drds\nonumber
\\
\leq &K\int^u_0E\Big(\sup_{0\leq r\leq s}(1+|X^n_r|^2)^{\frac{p}{2}}\big|\mathscr{F}_T^{\alpha}\Big)ds\label{eq:G_2}
\end{align}
almost surely for any $u\in[0,T]$ and $n\in\mathbb{N}$. The application of  Burkholder-Davis-Gundy inequality, Young's inequality,  H\"older's inequality and Corollary 3.1 implies, 
\begin{align}
G_3:=&pE\Big( \sup_{0\leq t\leq u}\big|\int^t_0(1+|X^n_s|^2)^{\frac{p-2}{2}}X^n_s\tilde{\sigma}(s,X^n_{\kappa(n,s)},\alpha_{\kappa(n,s)})dW_s\big| \Big|\mathscr{F}_T^{\alpha}\Big)\notag
\\
\leq&4p\sqrt{2}E\Big(\Big(\int^u_0(1+|X^n_s|^2)^{p-1}|\tilde{\sigma}(s,X^n_{\kappa(n,s)},\alpha_{\kappa(n,s)})|^2ds\Big)^{\frac{1}{2}} \Big|\mathscr{F}_T^{\alpha}\Big) \notag
\\
\leq&4p\sqrt{2}E\Big(\Big(\sup_{0\leq t \leq u}(1+|X^n_t|^2)^{p-1}\int^u_0|\tilde{\sigma}(s,X^n_{\kappa(n,s)},\alpha_{\kappa(n,s)})|^2ds\Big)^{\frac{1}{2}} \Big|\mathscr{F}_T^{\alpha}\Big) \notag
\\
\leq & \frac{1}{2}E\Big(\sup_{0\leq t \leq u}(1+|X^n_t|^2)^{\frac{p}{2}}\Big|\mathscr{F}_T^{\alpha}\Big)+KE\Big(\Big(\int^u_0|\tilde{\sigma}(s,X^n_{\kappa(n,s)},\alpha_{\kappa(n,s)})|^2ds\Big)^{\frac{p}{2}} \Big|\mathscr{F}_T^{\alpha}\Big) \notag
\\
\leq &\frac{1}{2}E\Big(\sup_{0\leq t \leq u}(1+|X^n_t|^2)^{\frac{p}{2}}\Big|\mathscr{F}_T^{\alpha}\Big)+KE\Big(\int^u_0|\tilde{\sigma}(s,X^n_{\kappa(n,s)},\alpha_{\kappa(n,s)})|^pds \Big|\mathscr{F}_T^{\alpha}\Big) \notag
\\
 \leq& \frac{1}{2}E\Big(\sup_{0\leq t \leq u}(1+|X^n_t|^2)^{\frac{p}{2}}\Big|\mathscr{F}_T^{\alpha}\Big)+K\int^u_0E\Big(\sup_{0\leq r \leq s}(1+|X^n_r|^2)^{\frac{p}{2}}\Big|\mathscr{F}_T^{\alpha}\Big)ds\label{eq:G_3Estimate}
\end{align}
almost surely for any $u\in[0,T]$ and $n\in\mathbb{N}$. Moreover, due to Young's inequality and Corollary \ref{cor:tildaSigma}, one obtains
\begin{align}
G_4:=&KE\Big(\int^u_0(1+|X^n_s|^2)^{\frac{p-2}{2}}|\tilde{\sigma}(s,X^n_{\kappa(n,s)},\alpha_{\kappa(n,s)})|^2ds\big|\mathscr{F}_T^{\alpha}\Big)\nonumber
\\
\leq &K\int^u_0E\Big(\sup_{0\leq r\leq s}(1+|X^n_r|^2)^{\frac{p}{2}}\big|\mathscr{F}_T^{\alpha}\Big)ds+K\int^u_0E\Big(|\tilde{\sigma}(s,X^n_{\kappa(n,s)},\alpha_{\kappa(n,s)})|^p\big|\mathscr{F}_T^{\alpha}\Big)ds\nonumber
\\
\leq &K\int^u_0E\Big(\sup_{0\leq r\leq s}(1+|X^n_r|^2)^{\frac{p}{2}}\big|\mathscr{F}_T^{\alpha}\Big)ds\label{eq:G_4}
\end{align}
almost surely for any $u\in[0,T]$ and $n\in\mathbb{N}$. On substituting estimates from \eqref{eq:G_1} to \eqref{eq:G_4} in \eqref{eq:G_splitting}, one obtains, 
\begin{align*}
E\Big(\sup_{0\leq t\leq u}(1+|X^n_t|^2)^{\frac{p}{2}}\big|\mathscr{F}_T^{\alpha}\Big)\leq& \frac{1}{2} E\Big(\sup_{0\leq t\leq u}(1+|X^n_t|^2)^{\frac{p}{2}}\big|\mathscr{F}_T^{\alpha}\Big) + E(1+|X_0^n|^2)^\frac{p}{2}
\\
& \quad+K\int^u_0E\Big(\sup_{0\leq r\leq s}(1+|X^n_r|^2)^{\frac{p}{2}}\big|\mathscr{F}_T^{\alpha}\Big)ds
\end{align*}
almost surely for any $u\in[0,T]$ and $n\in\mathbb{N}$. One finishes the proof by using Gronwall's lemma.
\end{proof}

\section{\bf{Proof of Main Result}.}
\noindent
The proof of main result \textit{i.e.}  Theorem \ref{thm:main} requires some useful lemmas which are  proved below. 
\begin{lemma}\label{lem:sigmaTermRate}
Let Assumptions $H$-2 to $H$-5 hold. Then,
\begin{align*}
E\Big(|\sigma_1(t,X^n_{\kappa(n,t)},\alpha_{\kappa(n,t)})|^2\big|\mathscr{F}_T^{\alpha}\Big)\leq&Kn^{-1}
\\
E\Big(|\sigma_2(t,X^n_{\kappa(n,t)},\alpha_{\kappa(n,t)})|^2\big|\mathscr{F}_T^{\alpha}\Big) &  \leq K \mathbb{I}\{N_{(\kappa(n,t),t)}=1\}\leq K
\end{align*}
almost surely for any $t\in[0,T]$ and $n\in \mathbb{N}$ where constant $K>0$ does not depend on $n$.
\end{lemma}
\begin{proof}
The proof of this lemma follows  from similar arguments as used in Lemma \ref{lem:diffu.bound} along with the application of Lemma \ref{lem:SchemeMoment}. 
\end{proof}

The following lemma is very useful in establishing the main result of this article. It can be regarded as a substitute for Lemma 2.2 of \cite{nguyen2017}  which is a kind of It\^o's formula for the function $g$ (see below). Later, in the poof of main result,  $g$ is taken to be  $b$ and $\sigma$.   Hence, it plays an important role in reducing the regularity requirements on the coefficients as It\^o's formula (i.e. Lemma 2.2 of \cite{nguyen2017}) requires second order derivatives of $g$.

\begin{lemma}\label{lem:for applying mvt}
Let $N_{(r,t)}$ be the number of jumps of the Markov chain $\alpha$ in the interval $(r,t)$ for any $0\leq r < t\leq T$. Also, suppose that  $\tau_1< \tau_2< \ldots< \tau_{N_{(r,t)}}$ are time of jumps of $\alpha$ in the interval $(r,t)$, where $t$ may or may not be the jump time. Let $\tau_0=r$ and $\tau_{N_{(r,t)}+1}=t$. Also, let $g(\cdot,i_0):\mathbb{R}^d \to \mathbb{R}^d$ be a measurable function for every $i_0 \in\mathcal{S}$. Then, almost surely, 
\begin{align*}
g(X_t,\alpha_t)&-g(X_r, \alpha_r)=\sum_{i_0\neq j_0}\int_r^t(g(X_u,j_0)-g(X_u,i_0))dM_{i_0j_0}(u)
\\
&\quad+\sum_{j_0\in\mathcal{S}}\int_r^t q_{\alpha_{u-}j_0}\big(g(X_u,j_0)-g(X_u,\alpha_{u-})\big)du + \sum_{k =0}^{N_{(r,t)}}\big( g(X_{\tau_{k+1}},\alpha_{\tau_{k}})-g(X_{\tau_{k}},\alpha_{\tau_{k}})  \big)
\end{align*}
for any $0\leq r<t\leq T$.  
\end{lemma}
\begin{proof}
Notice that  one can clearly write the following expression, almost surely, 
\begin{align}
g(X_t,\alpha_t)&-g(X_r, \alpha_r)=g(X_{\tau_{N_{(r,t)}+1}},\alpha_{\tau_{N_{(r,t)}+1})}-g(X_{\tau_{0}}, \alpha_{\tau_{0}}) =\sum_{k=0}^{N_{(r,t)}}\big( g(X_{\tau_{k+1}},\alpha_{\tau_{k+1}})-g(X_{\tau_{k}}, \alpha_{\tau_{k}})\big) \notag
\\
 =&\sum_{k=0}^{N_{(r,t)}}\big( g(X_{\tau_{k+1}},\alpha_{\tau_{k+1}})-g(X_{\tau_{k+1}},\alpha_{\tau_{k}})\big) + \sum_{k=0}^{N_{(r,t)}}\big(g(X_{\tau_{k+1}},\alpha_{\tau_{k}}) -g(X_{\tau_{k}}, \alpha_{\tau_{k}})\big) \label{eq:g}
\end{align}
for any  $0\leq r<t\leq T$. Moreover, the first term on the right hand side of the above equation can be written as,  
\begin{align}
&\sum_{i_0\neq j_0}\int_r^t  \big(g(X_u,j_0)-g(X_u,i_0)\big)dM_{i_0j_0}(u)
 =\sum_{i_0\neq j_0}\int_r^t \Big(g(X_u,j_0)-g(X_u,i_0)\Big)d[ M_{i_0j_0}](u) \notag
\\
&\quad- \sum_{i_0\neq j_0}\int_r^t \big(g(X_u,j_0)-g(X_u,i_0)\big)d\langle M_{i_0j_0}\rangle(u) \notag
\\
&=\sum_{k=0}^{N_{(r,t)}}\big( g(X_{\tau_{k+1}},\alpha_{\tau_{k+1}})-g(X_{\tau_{k+1}},\alpha_{\tau_{k}})\big) - \sum_{i_0\neq j_0} \int_r^t q_{i_0j_0}\mathbb{I}\{\alpha_{u-}=i_0\}\big(g(X_u,j_0)-g(X_u,i_0)\big)du \notag
\\
&=\sum_{k=0}^{N_{(r,t)}}\big( g(X_{\tau_{k+1}},\alpha_{\tau_{k+1})}-g(X_{\tau_{k+1}},\alpha_{\tau_{k}})\big) - \sum_{ j_0\in\mathcal{S}} \int_r^t q_{\alpha_{u-} j_0}\big(g(X_u,j_0)-g(X_u,\alpha_{u-})\big)du \label{eq:g1}
\end{align}
almost surely for  $0\leq r<t\leq T$. The proof is completed by substituting \eqref{eq:g1} in \eqref{eq:g}. 
\end{proof}
In addition to the above Lemma  \ref{lem:for applying mvt}, one also requires the following lemma to reduce the regularity requirements on the coefficients as its proof only depends once differentiability of $f$.  Later, in the proof of main result, $f$ is taken to $b$ and $\sigma$. 

\begin{lemma}\label{lem: mvt}
Let $f(\cdot,i_0):\mathbb{R}^d \rightarrow \mathbb{R}^d$ be a continuously differentiable function and satisfies, for every $i_0\in\mathcal{S}$,  
\begin{align} \label{eq:hyp}
|\mathcal{D}f(x,i_0)-\mathcal{D}f(\tilde{x},i_0)|\leq K(1+|x|+|\tilde{x}|)^{\gamma} |x-\tilde{x}| 
\end{align}
for any $x,\tilde{x}\in \mathbb{R}^d$ and for a fixed $\gamma\in \mathbb{R}$. Then, for every $i_0\in\mathcal{S}$,
\begin{equation*}
|f(x,i_0)-f(\tilde{x},i_0)-\mathcal{D} f(\tilde{x},i_0)(x-\tilde{x})|\leq K(1+|x|+|\tilde{x}|)^{\gamma}|x-\tilde{x}|^2
\end{equation*}
for any $x,\tilde{x}\in \mathbb{R}^d$. In the above, $K>0$ is  constant. 
\end{lemma}
\begin{proof}
For every $i_0\in\mathcal{S}$, due to mean value theorem, 
\begin{align*}
f(x,i_0)-f(\tilde{x}, i_0)=\mathcal{D} f(qx+(1-q)\tilde{x},i_0)(x-\tilde{x})
\end{align*}
for some $q\in (0,1)$ which on using  hypothesis  \eqref{eq:hyp}  further implies, 
\begin{align*}
|f(x,i_0)&-f(\tilde{x},i_0)-\mathcal{D} f(\tilde{x},i_0)(x-\tilde{x})|
\\
=&\Big|\mathcal{D} f(qx+(1-q)\tilde{x},i_0)(x-\tilde{x})- \mathcal{D} f(\tilde{x},i_0)(x-\tilde{x}) \Big|
\\
\leq & K(1+|qx+(1-q)\tilde{x}|+|\tilde{x}|)^{\gamma}|x-\tilde{x}|^2 \leq  K(1+|x|+|\tilde{x}|)^{\gamma}|x-\tilde{x}|^2 
\end{align*}
for any $x,\tilde{x}\in \mathbb{R}^d$. This completes the proof. 
\end{proof}
The proof of parts (a) and (b) of the following lemma can be found in \cite{nguyen2017}. For the completeness, their proofs are given below along with that of part (c).  
\begin{lemma}\label{lem:EN_i, EN_i^2 }
Let $N_{(\kappa(n,t),t)}$ be the number of jumps of the Markov chain $\alpha$ in the interval $(\kappa(n,t), t)$ for any $t\in[0,T]$. Set $q:=\max \{-q_{i_0i_0}; i_0 \in \mathcal{S}\}$.  Then, 
\newline
(a).  $P(N_{(\kappa(n,t),t)}\geq N)\leq q^Nh^N$ whenever $N\geq 1$,
\newline
(b). $EN_{(\kappa(n,t),t)} \leq K h$ whenever $h<1/(2q)$, where constant $K>0$ is  independent of $h$, and 
\newline
(c). $ EN_{(\kappa(n,t),t)}^2 \leq 6$. 
\end{lemma}
\begin{proof}
Let us denote by  $\tau_1,\ldots,\tau_{N_{(\kappa(n,t),t)}}$ the jump-times of the Markov chain $\alpha$ in  $(\kappa(n,t),t)$. Set $\tau_0=\kappa(n,t)$ and $\tau_{N_{(\kappa(n,t),t)}+1}=t$. Obviously, on the set  $\{N_{(\kappa(n,t),t)}\geq 1\}$, time between successive jumps i.e. $\tau_1-\tau_0$, $\tau_2-\tau_1$, $\tau_3-\tau_2$, $\ldots$, $\tau_{N_{(\kappa(n,t),t)}-1}-\tau_{N_{(\kappa(n,t),t)}}$, $\tau_{N_{(\kappa(n,t),t)}+1}-\tau_{N_{(\kappa(n,t),t)}}$ are conditionally independent random variables. Also, if $N_{(\kappa(n,t),t)}\geq 1$ and the Markov chain jumps from state $i_{r-1}$ to $i_r$ at time $\tau_r$, then $\tau_{r+1}-\tau_r$ follows  exponential distribution with parameter $-q_{i_ri_r}$ for $r=1,\ldots,N_{(\kappa(n,t),t)}$. Therefore,  strong Markov property of $\alpha$ implies that,  for any $N\geq 1$ and for any $t\in[0,T]$
\begin{align*}
P(N_{(\kappa(n,t),t)}\geq N)& \leq P\Big(\sum_{r=0}^{N-1} (\tau_{r+1}-\tau_r)<h\Big)\leq \prod_{r=0}^{N-1}P(\tau_{r+1}-\tau_r<h) 
\\
& \leq \prod_{r=0}^{N-1}(1-e^{q_{i_r i_r}h}) \leq \prod_{r=0}^{N-1} (-q_{i_r i_r}h) \leq q^Nh^N. 
\end{align*}
This proves part (a). As a consequence of part (a),   for any $t\in[0,T]$, 
\begin{align*}
EN_{(\kappa(n,t),t)}=\sum_{N=1}^\infty P(N_{(\kappa(n,t),t)}\geq N) \leq \sum_{N=1}^\infty q^Nh^N \leq q h  \sum_{N=0}^\infty (1/2)^N \leq K h. 
\end{align*}    
This shows pat (b). Moreover, for any $t\in[0,T]$, 
\begin{align*}
EN_{(\kappa(n,t),t)}^2=&\sum_{N=1}^\infty  N^2 P(N_{(\kappa(n,t),t)}=N)\leq \sum_{N=1}^\infty N^2 P(N_{(\kappa(n,t),t)}\geq N) \leq  \sum_{N=1}^\infty N^2 q^N h^N 
\\
\leq& \sum_{N=1}^\infty N^2 (1/2)^N =6.
\end{align*}
This proves part (c). 
\end{proof}
The following lemma holds for any $p_0\geq 2$ and later it will be used for $p_0=2,4,8$. 
\begin{lemma}\label{lem:scheme_differ_rate}
Let Assumptions $H$-2 to $H$-5 be satisfied.  Then, for any $\mathscr{F}_T^\alpha$ -measurable positive random variables $\tau_1$ and $\tau_2$ satisfying $\tau_1<\tau_2$ almost surely, the following holds,  
\begin{equation}
E\Big(\sup_{0\leq t \leq T}|X^n_{t\wedge\tau_2}-X^n_{t\wedge \tau_1}|^{p_0}\Big|\mathscr{F}_T^{\alpha}\Big)\leq K(\tau_2-\tau_1)^{\frac{p_0}{2}}+K(\tau_2-\tau_1)^{p_0}  \notag
\end{equation}
almost surely for any $n\in \mathbb{N}$.
\end{lemma}
\begin{proof}
The application of H\"older's inequality, Burkholder-Davis-Gundy inequality, Assumption $H$-5, Corollary \ref{cor:tildaSigma} and Lemma \ref{lem:SchemeMoment}   implies,
\begin{align*}
E\Big(\sup_{0\leq t \leq T}|X^n_{t\wedge\tau_2}-X^n_{t\wedge \tau_1}|^{p_0}&\Big|\mathscr{F}_T^{\alpha}\Big)\leq  KE\Big(\sup_{0\leq t \leq T}\Big|\int^{t\wedge \tau_2}_{t\wedge \tau_1}b^n(X^n_{\kappa(n,s)},\alpha_{\kappa(n,s)})ds\Big|^{p_0}\Big|\mathscr{F}_T^{\alpha}\Big)
\\
&\quad+KE\Big(\sup_{0\leq t \leq T}\Big|\int^{t\wedge \tau_2}_{t\wedge \tau_1}\tilde{\sigma}(s,X^n_{\kappa(n,s)},\alpha_{\kappa(n,s)})dW_s\Big|^{p_0}\Big|\mathscr{F}_T^{\alpha}\Big)
\\
\leq &K(\tau_2- \tau_1)^{p_0-1}\int^{\tau_2}_{ \tau_1}E\Big(|b^n(X^n_{\kappa(n,s),\alpha_{\kappa(n,s)}})|^{p_0}\Big|\mathscr{F}_T^{\alpha}\Big)ds
\\
&\quad+K( \tau_2-\tau_1)^{\frac{p_0}{2}-1}\int^{\tau_2}_{\tau_1}E\Big(|\tilde{\sigma}(s,X^n_{\kappa(n,s)},\alpha_{\kappa(n,s)})|^{p_0}\Big|\mathscr{F}_T^{\alpha}\Big)ds
\\
\leq & K(\tau_2-\tau_1)^{\frac{p_0}{2}}+K(\tau_2-\tau_1)^{p_0}
\end{align*}
almost surely for any $n\in \mathbb{N}$.
\end{proof}
\begin{lemma}\label{lem:sigmaCrusial}
Let Assumptions $H$-2 to $H$-5 be satisfied. Then,
\begin{align*}
\sup_{0\leq t \leq T}E\big|\sigma(X^n_t,\alpha_t)-\tilde{\sigma}(t,X^n_{\kappa(n,t)},\alpha_{\kappa(n,t)})\big|^2\leq Kn^{-2}
\end{align*}
for any $n\in \mathbb{N}$, where $0<h=(1/n)<1/(2q)$ with $q=\max \{-q_{i_0i_0}; i_0 \in \mathcal{S}\}$.
\end{lemma}
\begin{proof}
Due to Lemma \ref{lem:for applying mvt} along with $M_{i_0j_0}(t)= [M_{i_0 j_0}](t)-\langle M_{i_0j_0}\rangle(t)$, for any $i_0,j_0\in \mathcal{S}$ and $t\in[0,T]$, one obtains
\begin{align*}
\sigma^{(l)}&(X^n_t,\alpha_t)-\sigma^{(l)}(X^n_{\kappa(n,t)},\alpha_{\kappa(n,t)})=\sum_{i_0\neq j_0}\int^t_{\kappa(n,t)}\Big(\sigma^{(l)}(X^n_u,j_0)-\sigma^{(l)}(X^n_u,i_0)\Big)dM_{i_0j_0}(u)\nonumber
\\
&\quad+\sum_{j_0\in \mathcal{S}}\int^t_{\kappa(n,t)}q_{\alpha_{u-}j_0}\Big(\sigma^{(l)}(X^n_u,j_0)-\sigma^{(l)}(X^n_u,\alpha_{u-})\Big)du\nonumber
\\
&\quad+\sum^{N_{(\kappa(n,t),t)}}_{k=0}\Big( \sigma^{(l)}(X^n_{\tau_{k+1}},\alpha_{\tau_k})-\sigma^{(l)}(X^n_{\tau_k},\alpha_{\tau_k}) \Big)\nonumber
\\
=&\sum_{i_0\neq j_0}\int^t_{\kappa(n,t)}\Big(\sigma^{(l)}(X^n_u,j_0)-\sigma^{(l)}(X^n_u,i_0)\Big)d[M_{i_0j_0}](u)\nonumber
\\
&\quad+\sum_{i_0\neq j_0}\int^t_{\kappa(n,t)}\Big(\sigma^{(l)}(X^n_u,i_0)-\sigma^{(l)}(X^n_u,j_0)\Big)d\langle M_{i_0j_0}\rangle (u)\nonumber
\\
&\quad+\sum_{j_0\in \mathcal{S}}\int^t_{\kappa(n,t)}q_{\alpha_{u-}j_0}\Big(\sigma^{(l)}(X^n_u,j_0)-\sigma^{(l)}(X^n_u,\alpha_{u-})\Big)du\nonumber
\\
&\quad+\sum^{N_{(\kappa(n,t),t)}}_{k=0}\Big( \sigma^{(l)}(X^n_{\tau_{k+1}},\alpha_{\tau_k})-\sigma^{(l)}(X^n_{\tau_k},\alpha_{\tau_k})-\mathcal{D}\sigma^{(l)}(X^n_{\tau_k},\alpha_{\tau_k})(X^n_{\tau_{k+1}}-X^n_{\tau_k}) \Big)\nonumber
\\
&\quad+\sum^{N_{(\kappa(n,t),t)}}_{k=0}\int^{\tau_{k+1}}_{\tau_k}\mathcal{D}\sigma^{(l)}(X^n_{\tau_k},\alpha_{\tau_k})b^n(X^n_{\kappa(n,u)},\alpha_{\kappa(n,u)})du\nonumber
\\
&\quad+\sum_{l_1=1}^m\sum^{N_{(\kappa(n,t),t)}}_{k=0}\int^{\tau_{k+1}}_{\tau_k}\mathcal{D}\sigma^{(l)}(X^n_{\tau_k},\alpha_{\tau_k})\tilde{\sigma}^{(l_1)}(u,X^n_{\kappa(n,u)},\alpha_{\kappa(n,u)})dW_u^{l_1}\nonumber
\\
=&\mathbb{I}\{ N_{(\kappa(n,t),t)}=1 \}\sum_{i_0\neq j_0}\int^t_{\kappa(n,t)}\Big(\big(\sigma^{(l)}(X^n_u,j_0)-\sigma^{(l)}(X^n_u,i_0)\big)\notag
\\
&\quad\qquad-\big(\sigma^{(l)}(X^n_{\kappa(n,u)},j_0)-\sigma^{(l)}(X^n_{\kappa(n,u)},i_0)\big)\Big)d[M_{i_0j_0}](u)\nonumber
\\
&+\mathbb{I}\{ N_{(\kappa(n,t),t)}=1 \}\sum_{i_0\neq j_0}\int^t_{\kappa(n,t)}\Big(\sigma^{(l)}(X^n_{\kappa(n,u)},j_0)-\sigma^{(l)}(X^n_{\kappa(n,u)},i_0)\Big)d[M_{i_0j_0}](u)\nonumber
\\
&+\mathbb{I}\{ N_{(\kappa(n,t),t)}\geq2 \}\sum_{i_0\neq j_0}\int^t_{\kappa(n,t)}\Big(\sigma^{(l)}(X^n_u,j_0)-\sigma^{(l)}(X^n_u,i_0)\Big)d[M_{i_0j_0}](u)\nonumber
\\
&+\sum_{i_0\neq j_0}\int^t_{\kappa(n,t)}\Big(\sigma^{(l)}(X^n_u,i_0)-\sigma^{(l)}(X^n_u,j_0)\Big)d\langle M_{i_0j_0}\rangle (u)\nonumber
\\
&+\sum_{j_0\in \mathcal{S}}\int^t_{\kappa(n,t)}q_{\alpha_{u-}j_0}\Big(\sigma^{(l)}(X^n_u,j_0)-\sigma^{(l)}(X^n_u,\alpha_{u-})\Big)du\nonumber
\\
&+\sum^{N_{(\kappa(n,t),t)}}_{k=0}\Big( \sigma^{(l)}(X^n_{\tau_{k+1}},\alpha_{\tau_k})-\sigma^{(l)}(X^n_{\tau_k},\alpha_{\tau_k})-\mathcal{D}\sigma^{(l)}(X^n_{\tau_k},\alpha_{\tau_k})(X^n_{\tau_{k+1}}-X^n_{\tau_k}) \Big)\nonumber
\\
&+\sum^{N_{(\kappa(n,t),t)}}_{k=0}\int^{\tau_{k+1}}_{\tau_k}\mathcal{D}\sigma^{(l)}(X^n_{\tau_k},\alpha_{\tau_k})b^n(X^n_{\kappa(n,u)},\alpha_{\kappa(n,u)})du\nonumber
\\
&+\sum^m_{l_1=1}\sum^{N_{(\kappa(n,t),t)}}_{k=0}\int^{\tau_{k+1}}_{\tau_k}\Big( \mathcal{D}\sigma^{(l)}(X^n_{\tau_k},\alpha_{\tau_k})-\mathcal{D}\sigma^{(l)}(X^n_{\kappa(n,u)},\alpha_{\kappa(n,u)}) \Big)\sigma^{(l_1)}(X^n_{\kappa(n,u)},\alpha_{\kappa(n,u)})dW_u^{l_1}\nonumber
\\
&+\sum^m_{l_1=1}\sum^{N_{(\kappa(n,t),t)}}_{k=0}\int^{\tau_{k+1}}_{\tau_k}\mathcal{D}\sigma^{(l)}(X^n_{\kappa(n,u)},\alpha_{\kappa(n,u)})\sigma^{(l_1)}(X^n_{\kappa(n,u)},\alpha_{\kappa(n,u)})dW_u^{l_1}\nonumber
\\
&+\sum^m_{l_1=1}\sum^{N_{(\kappa(n,t),t)}}_{k=0}\int^{\tau_{k+1}}_{\tau_k}\mathcal{D}\sigma^{(l)}(X^n_{\tau_k},\alpha_{\tau_k})\Big(\sigma^{(l_1)}_1(X^n_{\kappa(n,u)},\alpha_{\kappa(n,u)})+\sigma^{(l_1)}_2(X^n_{\kappa(n,u)},\alpha_{\kappa(n,u)})\Big)dW_u^{l_1}
\end{align*}
almost surely for any $l=1,\ldots,m$ and $n\in \mathbb{N}$. Notice that second and ninth terms on the right side of preceding equation are $\sigma_2^{(l)}(t,X^n_{\kappa(n,t)},\alpha_{\kappa(n,t)})$ and $\sigma_1^{(l)}(t,X^n_{\kappa(n,t)},\alpha_{\kappa(n,t)})$ respectively. Therefore, preceding equation can be written as, 
\begin{align}
&E\big|\sigma^{(l)}(X^n_t,\alpha_t)-\tilde{\sigma}^{(l)}(t,X^n_{\kappa(n,t)},\alpha_{\kappa(n,t)})\big|^2\nonumber
\\
&\leq KE\Big|\mathbb{I}\{ N_{(\kappa(n,t),t)}=1 \}\sum_{i_0\neq j_0}\int^t_{\kappa(n,t)}\Big(\big(\sigma^{(l)}(X^n_u,j_0)-\sigma^{(l)}(X^n_u,i_0)\big)\notag
\\
&\quad\qquad-\big(\sigma^{(l)}(X^n_{\kappa(n,u)},j_0)-\sigma^{(l)}(X^n_{\kappa(n,u)},i_0)\big)\Big)d[M_{i_0j_0}](u)\Big|^2\notag
\\ 
&+ KE\Big|\mathbb{I}\{ N_{(\kappa(n,t),t)}\geq2 \}\sum_{i_0\neq j_0}\int^t_{\kappa(n,t)}\Big(\sigma^{(l)}(X^n_u,j_0)-\sigma^{(l)}(X^n_u,i_0)\Big)d[M_{i_0j_0}](u) \Big|^2\nonumber
\\
&+KE\Big|\sum_{i_0\neq j_0}\int^t_{\kappa(n,t)}\Big(\sigma^{(l)}(X^n_u,i_0)-\sigma^{(l)}(X^n_u,j_0)\Big)d\langle M_{i_0j_0}\rangle(u)\nonumber \Big|^2
\\
&+KE\Big|\sum_{j_0\in \mathcal{S}}\int^t_{\kappa(n,t)}q_{\alpha_{u-}j_0}\Big(\sigma^{(l)}(X^n_u,j_0)-\sigma^{(l)}(X^n_u,\alpha_{u-})\Big)du \Big|^2\nonumber
\\
&+KE\Big|\sum^{N_{(\kappa(n,t),t)}}_{k=0}\Big( \sigma^{(l)}(X^n_{\tau_{k+1}},\alpha_{\tau_k})-\sigma^{(l)}(X^n_{\tau_k},\alpha_{\tau_k})-\mathcal{D}\sigma^{(l)}(X^n_{\tau_k},\alpha_{\tau_k})(X^n_{\tau_{k+1}}-X^n_{\tau_k}) \Big) \Big|^2\nonumber
\\
&+KE\Big|\sum^{N_{(\kappa(n,t),t)}}_{k=0}\int^{\tau_{k+1}}_{\tau_k}\mathcal{D}\sigma^{(l)}(X^n_{\tau_k},\alpha_{\tau_k})b^n(X^n_{\kappa(n,u)},\alpha_{\kappa(n,u)})du \Big|^2\nonumber
\\
&+KE\Big|\sum^m_{l_1=1}\sum^{N_{(\kappa(n,t),t)}}_{k=0}\int^{\tau_{k+1}}_{\tau_k}\Big( \mathcal{D}\sigma^{(l)}(X^n_{\tau_k},\alpha_{\tau_k})-\mathcal{D}\sigma^{(l)}(X^n_{\kappa(n,u)},\alpha_{\kappa(n,u)}) \Big)\nonumber
\\
&\quad\qquad\times\sigma^{(l_1)}(X^n_{\kappa(n,u)},\alpha_{\kappa(n,u)})dW_u^{l_1} \Big|^2\nonumber
\\
&+KE\Big|\sum^m_{l_1=1}\sum^{N_{(\kappa(n,t),t)}}_{k=0}\int^{\tau_{k+1}}_{\tau_k}\mathcal{D}\sigma^{(l)}(X^n_{\tau_k},\alpha_{\tau_k})\Big(\sigma^{(l_1)}_1(X^n_{\kappa(n,u)},\alpha_{\kappa(n,u)})+\sigma^{(l_1)}_2(X^n_{\kappa(n,u)},\alpha_{\kappa(n,u)})\Big)dW_u^{l_1} \Big|^2\nonumber
\\
&=:T_1+T_2+T_3+T_4+T_5+T_6+T_7+T_8\label{eq:T_Splitting}
\end{align}
for any $l=1,\ldots,m$, $t\in[0,T]$ and $n\in \mathbb{N}$. Recall the    definition of $[M_{i_0j_0}]$ and use Assumption $H$-2 along with Lemmas [\ref{lem:EN_i, EN_i^2 },  \ref{lem:scheme_differ_rate}], to obtain the following,
\begin{align}
T_1:=&KE\Big|\mathbb{I}\{ N_{(\kappa(n,t),t)}=1 \}\sum_{i_0\neq j_0}\int^t_{\kappa(n,t)}\Big(\big(\sigma^{(l)}(X^n_u,j_0)-\sigma^{(l)}(X^n_u,i_0)\big)\notag
\\
&\quad-\big(\sigma^{(l)}(X^n_{\kappa(n,u)},j_0)-\sigma^{(l)}(X^n_{\kappa(n,u)},i_0)\big)\Big)d[M_{i_0j_0}](u)\Big|^2\notag
\\
\leq &KE\Big(\mathbb{I}\{ N_{(\kappa(n,t),t)}=1 \}E\Big\{\Big(\sum_{i_0\neq j_0}\int^t_{\kappa(n,t)}\big|\big(\sigma^{(l)}(X^n_u,j_0)-\sigma^{(l)}(X^n_u,i_0)\big)\notag
\\
&\quad-\big(\sigma^{(l)}(X^n_{\kappa(n,u)},j_0)-\sigma^{(l)}(X^n_{\kappa(n,u)},i_0)\big)\big|d[M_{i_0j_0}](u)\Big)^2\Big|\mathscr{F}_T^{\alpha} \Big\}   \Big)\notag
\\
\leq &KE\Big(\mathbb{I}\{ N_{(\kappa(n,t),t)}=1 \}E\Big\{\Big(\sum_{i_0\neq j_0}\int^t_{\kappa(n,t)}|X^n_u-X^n_{\kappa(n,u)}|d[M_{i_0j_0}](u)\Big)^2\Big|\mathscr{F}_T^{\alpha} \Big\}   \Big)\notag
\\
\leq &KE\Big(\mathbb{I}\{ N_{(\kappa(n,t),t)}=1 \}N^2_{(\kappa(n,t),t)}E\Big\{\sup_{0\leq u\leq t}|X^n_u-X^n_{\kappa(n,u)}|^2\Big|\mathscr{F}_T^{\alpha} \Big\}   \Big)\notag
\\
\leq&Kn^{-1}E\big(\mathbb{I}\{ N_{(\kappa(n,t),t)}=1 \}\big)\leq Kn^{-1}P(N_{(\kappa(n,t),t)}\geq 1)\leq Kn^{-2}\label{eq:T_1Rate}
\end{align}
for any $l=1,\ldots,m$, $t\in[0,T]$ and $n\in \mathbb{N}$. Again recall the definition of $[M_{i_0j_0}]$ and use Remark \ref{rem:growth} along with Lemmas [\ref{lem:SchemeMoment}, \ref{lem:EN_i, EN_i^2 }], $T_2$ can be estimated as, 
\begin{align}
T_2:=&KE\Big| \mathbb{I}\{ N_{(\kappa(n,t),t)}\geq2 \}\sum_{i_0\neq j_0}\int^t_{\kappa(n,t)}\Big(\sigma^{(l)}(X^n_u,j_0)-\sigma^{(l)}(X^n_u,i_0)\Big)d[M_{i_0j_0}](u) \Big|^2\notag
\\
\leq& KE\Big(\mathbb{I}\{ N_{(\kappa(n,t),t)}\geq2 \}E\Big\{  \Big(\sum_{i_0\neq j_0}\int^t_{\kappa(n,t)}\big|\sigma^{(l)}(X^n_u,j_0)-\sigma^{(l)}(X^n_u,i_0)\big|d[M_{i_0j_0}](u)\Big)^2\Big|\mathscr{F}_T^{\alpha} \Big\}    \Big)\notag
\\
\leq&KE\Big(\mathbb{I}\{ N_{(\kappa(n,t),t)}\geq2 \}\Big( \sum_{i_0\neq j_0}([M_{i_0j_0}](t)-[M_{i_0j_0}](\kappa(n,t))) \Big)^2E\Big\{ (1+\sup_{0\leq u\leq T}|X^n_u|^2)\Big|\mathscr{F}_T^{\alpha} \Big\}   \Big)\notag
\\
\leq& KE\Big(\mathbb{I}\{ N_{(\kappa(n,t),t)}\geq2 \}N_{(\kappa(n,t),t)}^2\Big)=K\sum_{N=2}^{\infty}N^2 P(N_{(\kappa(n,t),t)}=N)\nonumber
\\
\leq& K\sum_{N=2}^{\infty}N^2 P(N_{(\kappa(n,t),t)}\geq N)\leq K\sum_{N=2}^{\infty}N^2(qh)^N=Kq^2h^2\sum_{N=0}^{\infty}(N+2)^2(qh)^N\leq  Kn^{-2}\label{eq:T_2Rate}
\end{align}
for any $l=1,\ldots,m$, $t\in[0,T]$ and $n\in \mathbb{N}$. Moreover, H\"older's inequality, definition of $\langle M_{i_0j_0}\rangle$, Remark \ref{rem:growth} and Lemma \ref{lem:SchemeMoment}, yields
\begin{align}
T_3+T_4:=&KE\Big|\sum_{i_0\neq j_0}\int^t_{\kappa(n,t)}\Big(\sigma^{(l)}(X^n_u,i_0)-\sigma^{(l)}(X^n_u,j_0)\Big)d\langle M_{i_0j_0}\rangle (u)  \Big|^2\nonumber
\\
&\quad+KE\Big| \sum_{j_0\in \mathcal{S}}\int^t_{\kappa(n,t)}q_{\alpha_{u-}j_0}\Big(\sigma^{(l)}(X^n_u,j_0)-\sigma^{(l)}(X^n_u,\alpha_{u-})\Big)du \Big|^2\nonumber
\\
\leq& KE\sum_{i_0\neq j_0}\Big| \int^t_{\kappa(n,t)}\Big(\sigma^{(l)}(X^n_u,i_0)-\sigma^{(l)}(X^n_u,j_0)\Big)q_{i_0j_0}\mathbb{I}\{\alpha_{u-}=i_0 \} du\Big|^2\nonumber
\\
&\quad+KE\sum_{j_0\in \mathcal{S}}\Big|\int^t_{\kappa(n,t)}\Big(\sigma^{(l)}(X^n_u,j_0)-\sigma^{(l)}(X^n_u,\alpha_{u-})\Big)du\Big|^2\nonumber
\\
\leq& Kn^{-1}E\Big(\sum_{i_0\neq j_0} \int^t_{\kappa(n,t)}E\Big\{\big|(\sigma^{(l)}(X^n_u,i_0)-\sigma^{(l)}(X^n_u,j_0))\big|^2\Big|\mathscr{F}_T^{\alpha} \Big\} du  \Big)\nonumber
\\
&\quad+Kn^{-1}E\Big(\sum_{j_0\in \mathcal{S}}\int^t_{\kappa(n,t)}E\Big\{  \big|\sigma^{(l)}(X^n_u,j_0)-\sigma^{(l)}(X^n_u,\alpha_{u-})\big|^2 \Big|\mathscr{F}_T^{\alpha} \Big\} du  \Big)\nonumber
\\
\leq& Kn^{-1}E\Big( \int^t_{\kappa(n,t)}E\Big\{ (1+\sup_{0\leq u \leq T}|X^n_u|^2)\Big|\mathscr{F}_T^{\alpha} \Big\}  du\Big)\leq Kn^{-2}\label{eq:T_3,T_4Rate}
\end{align}
for any $l=1,\ldots,m$, $t\in[0,T]$ and $n\in \mathbb{N}$. Now, $T_5$ and $T_6$ can be written as
\begin{align}
T_5&+T_6:=KE\Big| \sum^{N_{(\kappa(n,t),t)}}_{k=0}\Big( \sigma^{(l)}(X^n_{\tau_{k+1}},\alpha_{\tau_k})-\sigma^{(l)}(X^n_{\tau_k},\alpha_{\tau_k})-\mathcal{D}\sigma^{(l)}(X^n_{\tau_k},\alpha_{\tau_k})(X^n_{\tau_{k+1}}-X^n_{\tau_k}) \Big)  \Big|^2\nonumber
\\
&\quad+KE\Big| \sum^{N_{(\kappa(n,t),t)}}_{k=0}\int^{\tau_{k+1}}_{\tau_k}\mathcal{D}\sigma^{(l)}(X^n_{\tau_k},\alpha_{\tau_k})b^n(X^n_{\kappa(n,u)},\alpha_{\kappa(n,u)})du \Big|^2\nonumber
\\
\leq& KE\Big( (1+N_{(\kappa(n,t),t)}) \sum^{N_{(\kappa(n,t),t)}}_{k=0} E\Big\{\big|\sigma^{(l)}(X^n_{\tau_{k+1}},\alpha_{\tau_k})-\sigma^{(l)}(X^n_{\tau_k},\alpha_{\tau_k})\nonumber
\\
&\quad\quad\qquad-\mathcal{D}\sigma^{(l)}(X^n_{\tau_k},\alpha_{\tau_k})(X^n_{\tau_{k+1}}-X^n_{\tau_k})\big|^2  \Big|\mathscr{F}_T^{\alpha}  \Big\}  \Big)\nonumber
\\
&\quad+KE\Big( (1+N_{(\kappa(n,t),t)}) \sum^{N_{(\kappa(n,t),t)}}_{k=0} E\Big\{  \Big| \int^{\tau_{k+1}}_{\tau_k}\mathcal{D}\sigma^{(l)}(X^n_{\tau_k},\alpha_{\tau_k})b^n(X^n_{\kappa(n,u)},\alpha_{\kappa(n,u)})du \Big|^2\Big|\mathscr{F}_T^{\alpha}  \Big\} \Big)\nonumber
\end{align}
which on the application of H\"older's inequality, Remark \ref{rem:growth},  Assumptions $H$-3, $H$-5 and Lemmas [\ref{lem:SchemeMoment}, \ref{lem: mvt}, \ref{lem:EN_i, EN_i^2 }, \ref{lem:scheme_differ_rate}] gives,
\begin{align}
T_5&+T_6\leq KE\Big((1+ N_{(\kappa(n,t),t)}) \sum^{N_{(\kappa(n,t),t)}}_{k=0} E\Big\{\big|X^n_{\tau_{k+1}}-X^n_{\tau_k}\big|^4  \Big|\mathscr{F}_T^{\alpha}  \Big\}  \Big)\nonumber
\\
&+Kn^{-1}E\Big( (1+N_{(\kappa(n,t),t)}) \sum^{N_{(\kappa(n,t),t)}}_{k=0}  \int^{\tau_{k+1}}_{\tau_k}E\Big\{  \big|\mathcal{D}\sigma^{(l)}(X^n_{\tau_k},\alpha_{\tau_k})b^n(X^n_{\kappa(n,u)},\alpha_{\kappa(n,u)})\big|^2\Big|\mathscr{F}_T^{\alpha}  \Big\}du  \Big)\nonumber
\\
\leq& KE\Big((1+N_{(\kappa(n,t),t)})\sum^{N_{(\kappa(n,t),t)}}_{k=0}\big((\tau_{k+1}-\tau_k)^2+(\tau_{k+1}-\tau_k)^4\big)   \Big)\nonumber
\\
&+Kn^{-1}E\Big( (1+N_{(\kappa(n,t),t)}) \sum^{N_{(\kappa(n,t),t)}}_{k=0} \int^{\tau_{k+1}}_{\tau_k}E\Big\{ \sup_{0\leq u\leq T}(1+|X^n_{\kappa(n,u)}|)^{2\rho+2}\Big|\mathscr{F}_T^{\alpha}  \Big\}du  \Big)\nonumber
\\
\leq & K n^{-2} E(1+N_{(\kappa(n,t),t)})^2+Kn^{-1}E\Big( (1+N_{(\kappa(n,t),t)}) \sum^{N_{(\kappa(n,t),t)}}_{k=0} (\tau_{k+1}-\tau_k)\Big) \leq Kn^{-2}\label{eq:T_5T_6Rate}
\end{align}
for any $l=1,\ldots,m$, $t\in[0,T]$ and $n\in \mathbb{N}$. One uses Remark \ref{rem:growth} and Assumption $H$-3, $T_7$ can be estimated as, 
\begin{align*}
T_7:=&KE\Big|\sum^m_{l_1=1}\sum^{N_{(\kappa(n,t),t)}}_{k=0}\int^{\tau_{k+1}}_{\tau_k}\Big( \mathcal{D}\sigma^{(l)}(X^n_{\tau_k},\alpha_{\tau_k})-\mathcal{D}\sigma^{(l)}(X^n_{\kappa(n,u)},\alpha_{\kappa(n,u)}) \Big)
\\
&\quad\quad\qquad\times\sigma^{(l_1)}(X^n_{\kappa(n,u)},\alpha_{\kappa(n,u)})dW_u^{l_1} \Big|^2
\\
\leq&K\sum^m_{l_1=1}E\Big(\sum^{N_{(\kappa(n,t),t)}}_{k=0}\int^{\tau_{k+1}}_{\tau_k}E\Big\{\big| \mathcal{D}\sigma^{(l)}(X^n_{\tau_k},\alpha_{\tau_k})-\mathcal{D}\sigma^{(l)}(X^n_{\kappa(n,u)},\alpha_{\kappa(n,u)}) \big|^2
\\
&\quad\quad\qquad\times\big|\sigma^{(l_1)}(X^n_{\kappa(n,u)},\alpha_{\kappa(n,u)})\big|^2\Big|\mathscr{F}_T^{\alpha}\Big\}du \Big)
\\
\leq & K\sum^m_{l_1=1}E\Big(\sum^{N_{(\kappa(n,t),t)}}_{k=0}\int^{\tau_{k+1}}_{\tau_k}E\Big\{ |\mathcal{D}\sigma^{(l)}(X^n_{\tau_k},\alpha_{\tau_k})-\mathcal{D}\sigma^{(l)}(X^n_{\tau_k},\alpha_{\kappa(n,u)})|^2
\\
&\quad\quad\qquad\times|\sigma^{(l_1)}(X^n_{\kappa(n,u)},\alpha_{\kappa(n,u)})|^2\Big|\mathscr{F}_T^{\alpha} \Big\}du\Big)
\\
&\quad+K\sum^m_{l_1=1}E\Big(\sum^{N_{(\kappa(n,t),t)}}_{k=0}\int^{\tau_{k+1}}_{\tau_k}E\Big\{ |\mathcal{D}\sigma^{(l)}(X^n_{\tau_k},\alpha_{\kappa(n,u)})-\mathcal{D}\sigma^{(l)}(X^n_{\kappa(n,u)},\alpha_{\kappa(n,u)})|^2
\\
&\quad\quad\qquad\times|\sigma^{(l_1)}(X^n_{\kappa(n,u)},\alpha_{\kappa(n,u)})|^2\Big|\mathscr{F}_T^{\alpha} \Big\}du\Big)
\\
\leq&KE\Big(\sum^{N_{(\kappa(n,t),t)}}_{k=0}\int^{\tau_{k+1}}_{\tau_k}\mathbb{I}\{\alpha_{\tau_k}\neq\alpha_{\kappa(n,u)}\}E\Big\{(1+\sup_{0\leq r\leq T}|X^n_r|^2)\Big|\mathscr{F}_T^{\alpha} \Big\}du\Big)
\\
&+KE\Big(\sum^{N_{(\kappa(n,t),t)}}_{k=0}\int^{\tau_{k+1}}_{\tau_k}E\Big\{|X^n_{\tau_k}-X^n_{\kappa(n,u)}|^2(1+|X^n_{\kappa(n,u)}|)^2\Big|\mathscr{F}_T^{\alpha} \Big\}du\Big)
\end{align*}
for any $l=1,\ldots,m$, $t\in[0,T]$ and $n\in \mathbb{N}$. Notice that  there is no jump in $(\tau_k, \tau_{k+1})$, hence $\alpha_{\tau_k}=\alpha_u$ for any $u\in(\tau_k,\tau_{k+1})$. Then, due to H\"older's inequality, and Lemmas [\ref{lem:SchemeMoment}, \ref{lem:scheme_differ_rate}] gives
\begin{align}
T_7\leq &KE\Big(\int^t_{\kappa(n,t)}\mathbb{I}\{\alpha_u\neq\alpha_{\kappa(n,u)}\}du\Big)\nonumber
\\
&\quad+KE\Big(\sum^{N_{(\kappa(n,t),t)}}_{k=0}\int^{\tau_{k+1}}_{\tau_k}\Big[E\Big\{|X^n_{\tau_k}-X^n_{\kappa(n,u)}|^4\Big|\mathscr{F}_T^{\alpha} \Big\}E\Big\{(1+\sup_{0\leq u \leq T}|X^n_u|^4)\Big|\mathscr{F}_T^{\alpha} \Big\}\Big]^{\frac{1}{2}}du\Big)\nonumber
\\
\leq &K\int^t_{\kappa(n,t)} P(\alpha_u\neq \alpha_{\kappa(n,u)})du+KE\Big(\sum^{N_{(\kappa(n,t),t)}}_{k=0}\int^{\tau_{k+1}}_{\tau_k}\big((\tau_k-\kappa(n,u))^2+(\tau_k-\kappa(n,u))^4\big)^{\frac{1}{2}}du\Big)\nonumber
\\
\leq& K\int^t_{\kappa(n,t)}\Big(q_{\alpha_u\alpha_{\kappa(n,u)}}(u-\kappa(n,u))+o(u-\kappa(n,u))\Big)du+ Kn^{-1}E\Big(\sum^{N_{(\kappa(n,t),t)}}_{k=0}(\tau_{k+1}-\tau_k)\Big)\notag
\\
\leq& Kn^{-2}\label{eq:T_7Rate}
\end{align}
for any $l=1,\ldots,m$, $t\in[0,T]$ and $n\in \mathbb{N}$. Moreover,  Remark \ref{rem:growth} and Lemmas [\ref{lem:sigmaTermRate}, \ref{lem:EN_i, EN_i^2 }] implies
\begin{align}
T_8:=&KE\Big|\sum^m_{l_1=1}\sum^{N_{(\kappa(n,t),t)}}_{k=0}\int^{\tau_{k+1}}_{\tau_k}\mathcal{D}\sigma^{(l)}(X^n_{\tau_k},\alpha_{\tau_k})\Big(\sigma^{(l_1)}_1(X^n_{\kappa(n,u)},\alpha_{\kappa(n,u)})\nonumber
\\
&\quad\quad+\sigma^{(l_1)}_2(X^n_{\kappa(n,u)},\alpha_{\kappa(n,u)})\Big)dW_u^{l_1}\Big|^2\nonumber
\\
\leq&K\sum^m_{l_1=1}E\Big(\sum^{N_{(\kappa(n,t),t)}}_{k=0}E\Big\{\Big|\int^{\tau_{k+1}}_{\tau_k}\mathcal{D}\sigma^{(l)}(X^n_{\tau_k},\alpha_{\tau_k})\Big(\sigma^{(l_1)}_1(X^n_{\kappa(n,u)},\alpha_{\kappa(n,u)})\nonumber
\\
&\quad\quad+\sigma^{(l_1)}_2(X^n_{\kappa(n,u)},\alpha_{\kappa(n,u)})\Big)dW_u^{l_1}\Big|^2\Big|\mathscr{F}_T^{\alpha}  \Big\}  \Big)\nonumber
\\
\leq& K\sum^m_{l_1=1}E\Big(\sum^{N_{(\kappa(n,t),t)}}_{k=0}\int^{\tau_{k+1}}_{\tau_k}E\Big\{|\sigma^{(l_1)}_1(X^n_{\kappa(n,u)},\alpha_{\kappa(n,u)})|^2+|\sigma^{(l_1)}_2(X^n_{\kappa(n,u)},\alpha_{\kappa(n,u)})|^2\Big|\mathscr{F}_T^{\alpha}  \Big\}du  \Big)\nonumber
\\
\leq& Kn^{-1}E\sum^{N_{(\kappa(n,t),t)}}_{k=0}(\tau_{k+1}-\tau_{k})  +KE\Big(\sum^{N_{(\kappa(n,t),t)}}_{k=0}\int^{\tau_{k+1}}_{\tau_k}\mathbb{I}\{N_{(\kappa(n,u),u)}=1\}du  \Big)\nonumber
\\
\leq & Kn^{-2}+KE\Big(\int^t_{\kappa(n,t)}\mathbb{I}\{N_{(\kappa(n,u),u)}=1\}du\Big)=Kn^{-2}+K\int^t_{\kappa(n,t)}P\big(N_{(\kappa(n,u),u)}=1)du\nonumber
\\
\leq& Kn^{-2}\label{eq:T_8Rate}
\end{align}
for any $l=1,\ldots,m$, $t\in[0,T]$ and $n\in \mathbb{N}$. By combining the estimates from \eqref{eq:T_1Rate} to \eqref{eq:T_8Rate} in \eqref{eq:T_Splitting}, one completes the proof.
\end{proof}
Let us define $e^n_t:=X_t-X^n_t$ for any $t\in[0,T]$ and $n\in\mathbb{N}$.
\begin{lemma}\label{lem:bCrusial}
Let Assumptions H-1 to H-6 be satisfied. Then,
\begin{align*}
E\int^t_0 e^n_s \Big(b(X^n_s,\alpha_s)-b(X^n_{\kappa(n,s)},\alpha_{\kappa(n,s)})\Big)ds\leq K\int^t_0 E\Big(\sup_{0\leq r \leq s}|e^n_r|^2\Big)ds+Kn^{-2}
\end{align*}
for any $n\in \mathbb{N}$ and $t\in[0,T]$, where $0<h=(1/n)<1/(2q)$ with $q=\max\{-q_{i_0i_0};i_0\in\mathcal{S}\}$. 
\end{lemma}
\begin{proof}
By using Lemma \ref{lem:for applying mvt}, one obtains, 
\begin{align*}
b(X^n_s,\alpha_s)&-b(X^n_{\kappa(n,s)},\alpha_{\kappa(n,s)})=\sum_{i_0\neq j_0}\int^s_{\kappa(n,s)}\Big(b(X^n_r,j_0)-b(X^n_r,i_0)\Big)dM_{i_0j_0}(r)
\\
&+\sum_{j_0\in \mathcal{S}}\int^s_{\kappa(n,s)}q_{\alpha_{r-}j_0}\Big(b(X^n_r,j_0)-b(X^n_r,\alpha_{r-})\Big)dr
\\
&+\sum_{k=0}^{N_{(\kappa(n,s),s)}}\Big(b(X^n_{\tau_{k+1}},\alpha_{\tau_k})-b(X^n_{\tau_k},\alpha_{\tau_k})-\mathcal{D}b(X^n_{\tau_k},\alpha_{\tau_k})(X^n_{\tau_{k+1}}-X^n_{\tau_k})\Big)
\\
&+\sum_{k=0}^{N_{(\kappa(n,s),s)}}\int^{\tau_{k+1}}_{\tau_k}\mathcal{D}b(X^n_{\tau_k},\alpha_{\tau_k})b^n(X^n_{\kappa(n,r)}, \alpha_{\kappa(n,r)})dr
\\
&+\sum^m_{l=1}\sum_{k=0}^{N_{(\kappa(n,s),s)}}\int^{\tau_{k+1}}_{\tau_k}\mathcal{D}b(X^n_{\tau_k},\alpha_{\tau_k})\tilde{\sigma}^{(l)}(r,X^n_{\kappa(n,r)},\alpha_{\kappa(n,r)})dW_r^l
\end{align*}
almost surely for any $s\in[0,T]$ and $n\in \mathbb{N}$. In order to simplify our notations, let us define, 
\begin{align}
H_s:=&\sum_{i_0\neq j_0}\int^s_{\kappa(n,s)}\Big(b(X^n_r,j_0)-b(X^n_r,i_0)\Big)dM_{i_0j_0}(r) \notag
\\
R_s:=&\sum^m_{l=1}\sum_{k=0}^{N_{(\kappa(n,s),s)}}\int^{\tau_{k+1}}_{\tau_k}\mathcal{D}b(X^n_{\tau_k},\alpha_{\tau_k})\tilde{\sigma}^{(l)}(r,X^n_{\kappa(n,r)},\alpha_{\kappa(n,r)})dW_r^l \notag
\end{align}
and hence one can write the following, 
\begin{align}
E\int^t_0 e^n_s &\Big(b(X^n_s,\alpha_s)-b(X^n_{\kappa(n,s)},\alpha_{\kappa(n,s)})\Big)ds=E\int^t_0 (e^n_s-e^n_{\kappa(n,s)})(H_s+R_s)ds\notag
\\
&+E\int^t_0e^n_{\kappa(n,s)}(H_s+R_s)ds\notag
\\
&+E\int^t_0 e^n_s\sum_{j_0\in \mathcal{S}}\int^s_{\kappa(n,s)}q_{\alpha_{r-}j_0}\Big(b(X^n_r,j_0)-b(X^n_r,\alpha_{r-})\Big)drds\nonumber
\\
&+E\int^t_0 e^n_s\sum_{k=0}^{N_{(\kappa(n,s),s)}}\Big(b(X^n_{\tau_{k+1}},\alpha_{\tau_k})-b(X^n_{\tau_k},\alpha_{\tau_k})-\mathcal{D}b(X^n_{\tau_k},\alpha_{\tau_k})(X^n_{\tau_{k+1}}-X^n_{\tau_k})\Big)ds\nonumber
\\
&+E\int^t_0 e^n_s\sum_{k=0}^{N_{(\kappa(n,s),s)}}\int^{\tau_{k+1}}_{\tau_k}\mathcal{D}b(X^n_{\tau_k},\alpha_{\tau_k})b^n(X^n_{\kappa(n,r)},\alpha_{\kappa(n,r)})drds\nonumber
\\
=&E\int^t_0\int^s_{\kappa(n,s)}\big(b(X_r,\alpha_r)-b^n(X^n_{\kappa(n,r)},\alpha_{\kappa(n,r)})\big)dr(H_s+R_s)ds\notag
\\
&+E\int^t_0\int^s_{\kappa(n,s)}\big(\sigma(X_r,\alpha_r)-\tilde{\sigma}(r,X^n_{\kappa(n,r)},\alpha_{\kappa(n,r)})\big)dW_r(H_s+R_s)ds \notag
\\
&+E\int^t_0e^n_{\kappa(n,s)}(H_s+R_s) ds\notag
\\
&+E\int^t_0 e^n_s\sum_{j_0\in \mathcal{S}}\int^s_{\kappa(n,s)}q_{\alpha_{r-}j_0}\Big(b(X^n_r,j_0)-b(X^n_r,\alpha_{r-})\Big)drds\nonumber
\\
&+E\int^t_0 e^n_s\sum_{k=0}^{N_{(\kappa(n,s),s)}}\Big(b(X^n_{\tau_{k+1}},\alpha_{\tau_k})-b(X^n_{\tau_k},\alpha_{\tau_k})-\mathcal{D}b(X^n_{\tau_k},\alpha_{\tau_k})(X^n_{\tau_{k+1}}-X^n_{\tau_k})\Big)ds\nonumber
\\
&+E\int^t_0 e^n_s\sum_{k=0}^{N_{(\kappa(n,s),s)}}\int^{\tau_{k+1}}_{\tau_k}\mathcal{D}b(X^n_{\tau_k},\alpha_{\tau_k})b^n(X^n_{\kappa(n,r)},\alpha_{\kappa(n,r)})drds\nonumber
\\
=:&C_1+C_2+C_3+C_4+C_5+C_6\label{eq:C_Splitting}
\end{align}
for any $t\in[0,T]$ and $n\in \mathbb{N}$.  For estimating $C_1$, one uses following splitting , 
\begin{align*}
\big(b(X_r,\alpha_r)&-b^n(X^n_{\kappa(n,r)},\alpha_{\kappa(n,r)})\big)=\big(b(X_r,\alpha_r)-b(X^n_r,\alpha_r)\big)+\big(b(X^n_r,\alpha_r)-b(X^n_{\kappa(n,r)},\alpha_r)\big)
\\
&+\big(b(X^n_{\kappa(n,r)},\alpha_r)-b(X^n_{\kappa(n,r)},\alpha_{\kappa(n,r)})\big)+\big(b(X^n_{\kappa(n,r)},\alpha_{\kappa(n,r)})-b^n(X^n_{\kappa(n,r)},\alpha_{\kappa(n,r)})\big)
\end{align*}
 to obtain,  
 \begin{align*}
 C_1  & := E\int^t_0\int^s_{\kappa(n,s)}\big(b(X_r,\alpha_r)-b^n(X^n_{\kappa(n,r)},\alpha_{\kappa(n,r)})\big)dr(H_s+R_s)ds\notag
 \\
 &  =E\int^t_0 \int^s_{\kappa(n,s)}\big(b(X_r,\alpha_r)-b(X^n_r,\alpha_r)\big) dr(H_s+R_s)ds
 \\
 &+E\int^t_0 \int^s_{\kappa(n,s)}\big(b(X^n_r,\alpha_r)-b(X^n_{\kappa(n,r)},\alpha_r)\big) dr(H_s+R_s)ds
 \\
 &+E\int^t_0 \int^s_{\kappa(n,s)}\big(b(X^n_{\kappa(n,r)},\alpha_r)-b(X^n_{\kappa(n,r)},\alpha_{\kappa(n,r)})\big) dr(H_s+R_s)ds
 \\
 &+E\int^t_0 \int^s_{\kappa(n,s)}\big(b(X^n_{\kappa(n,r)},\alpha_{\kappa(n,r)})-b^n(X^n_{\kappa(n,r)},\alpha_{\kappa(n,r)})\big) dr(H_s+R_s)ds
 \end{align*}
 and by the application of  Remark \ref{rem:growth}, Young's inequality and H\"older's inequality, one gets 
\begin{align}
C_1\leq& K\int^t_0nE\int^s_{\kappa(n,s)}|e^n_r|^2drds+Kn^{-1}\int^t_0\int^s_{\kappa(n,s)}\Big\{E(1+|X_r|+|X_r^n|)^{4\rho}E|H_s+R_s|^4\Big\}^{\frac{1}{2}}drds\notag
\\
&+K\int^t_0\bigg\{E\Big|\int^s_{\kappa(n,s)}(b(X^n_r,\alpha_r)-b(X^n_{\kappa(n,r)},\alpha_r)) dr \Big|^2E|H_s+R_s|^2 \bigg\}^{\frac{1}{2}}ds\notag
\\
&+K\int^t_0\bigg\{E\Big| \int^s_{\kappa(n,s)}(b(X^n_{\kappa(n,r)},\alpha_r)-b(X^n_{\kappa(n,r)},\alpha_{\kappa(n,r)})) dr\Big|^2E|H_s+R_s|^2 \bigg\}^{\frac{1}{2}}ds\notag
\\
&+K\int^t_0\bigg\{E\Big|\int^s_{\kappa(n,s)}(b(X^n_{\kappa(n,r)},\alpha_{\kappa(n,r)})-b^n(X^n_{\kappa(n,r)},\alpha_{\kappa(n,r)})) dr \Big|^2E|H_s+R_s|^2 \bigg\}^{\frac{1}{2}}ds\label{eq:C_1}
\end{align}
for any $n\in\mathbb{N}$ and $t\in[0,T]$. Now, for any $p_1\geq 2$, 
\begin{align*}
E|H_s+& R_s|^{p_1} \leq  KE\Big|\sum_{i_0\neq j_0}\int^s_{\kappa(n,s)}\Big(b(X^n_r,j_0)-b(X^n_r,i_0)\Big)d[M_{i_0j_0}](r)\Big|^{p_1}
\\
&+KE\Big|\sum_{i_0\neq j_0}\int^s_{\kappa(n,s)}\Big(b(X^n_r,j_0)-b(X^n_r,i_0)\Big)d\langle M_{i_0j_0}\rangle(r)\Big|^{p_1}
\\
&+KE\Big|\sum^m_{l=1}\sum_{k=0}^{N_{(\kappa(n,s),s)}}\int^{\tau_{k+1}}_{\tau_k}\mathcal{D}b(X^n_{\tau_k},\alpha_{\tau_k})\tilde{\sigma}^{(l)}(r,X^n_{\kappa(n,r)},\alpha_{\kappa(n,r)})dW_r^l\Big|^{p_1}
\\
\leq&KE\Big(\mathbb{I}\{N_{(\kappa(n,s),s)}=1\}\sum_{i_0\neq j_0}\int^s_{\kappa(n,s)}\big|b(X^n_r,j_0)-b(X^n_r,i_0)\big|d[M_{i_0j_0}](r)\Big)^{p_1}
\\
&+KE\Big(\mathbb{I}\{N_{(\kappa(n,s),s)}\geq2\}\sum_{i_0\neq j_0}\int^s_{\kappa(n,s)}\big|b(X^n_r,j_0)-b(X^n_r,i_0)\big|d[M_{i_0j_0}](r)\Big)^{p_1}
\\
&+KE\Big|\sum_{ j_0\in\mathcal{S}}\int^s_{\kappa(n,s)}q_{\alpha_{r-}j_0}\Big(b(X^n_r,j_0)-b(X^n_r,\alpha_{r-})\Big)dr\Big|^{p_1}
\\
&+KE\Big(\sum^m_{l=1}\sum_{k=0}^{N_{(\kappa(n,s),s)}}E\Big\{\Big|\int^{\tau_{k+1}}_{\tau_k}\mathcal{D}b(X^n_{\tau_k},\alpha_{\tau_k})\tilde{\sigma}^{(l)}(r,X^n_{\kappa(n,r)},\alpha_{\kappa(n,r)})dW_r^l\Big|^{p_1}\Big|\mathscr{F}_T^{\alpha}\Big\}\Big)
\end{align*}
which on  the application of Remark \ref{rem:growth}, H\"older's inequality, Young's inequality, and Corollary \ref{cor:tildaSigma} implies,
\begin{align*}
E|H_s+&R_s|^{p_1}\leq KE\Big(\mathbb{I}\{N_{(\kappa(n,s),s)}=1\}\Big|\sum_{i_0\neq j_0}\int^s_{\kappa(n,s)}d[M_{i_0j_0}](r)\Big|^{p_1}E\Big\{\sup_{0\leq r \leq T}(1+|X_r^n|)^{p_1\rho+p_1}\Big|\mathscr{F}_T^{\alpha}\Big\}\Big)
\\
&+KE\Big(\mathbb{I}\{N_{(\kappa(n,s),s)}\geq2\}\Big|\sum_{i_0\neq j_0}\int^s_{\kappa(n,s)}d[M_{i_0j_0}](r)\Big|^{p_1}E\Big\{\sup_{0\leq r \leq T}(1+|X_r^n|)^{p_1\rho+p_1}\Big|\mathscr{F}_T^{\alpha}\Big\}\Big)
\\
&+Kn^{-p_1+1}\int^s_{\kappa(n,s)}E\Big(\sup_{0\leq r\leq T}\big(1+|X^n_r|\big)^{p_1\rho+p_1}\Big)dr
\\
&+Kn^{-\frac{p_1}{2}+1}E\Big(\sum^m_{l=1}\sum_{k=0}^{N_{(\kappa(n,s),s)}}\int^{\tau_{k+1}}_{\tau_k}E\Big\{\big|\mathcal{D}b(X^n_{\tau_k},\alpha_{\tau_k})\tilde{\sigma}^{(l)}(r,X^n_{\kappa(n,r)},\alpha_{\kappa(n,r)})\big|^{p_1}\Big|\mathscr{F}_T^{\alpha}\Big\}dr\Big)
\end{align*}
for any $s\in[0,T]$ and $n\in\mathbb{N}$. Due to  Young's inequality, Remark \ref{rem:growth} and Lemmas [\ref{lem:SchemeMoment}, \ref{lem:EN_i, EN_i^2 }], one obtains, 
\begin{align}
E|H_s+R_s|^{p_1}\leq &KE\Big( \mathbb{I}\{N_{(\kappa(n,s),s)}=1\}N^{p_1}_{(\kappa(n,s),s)}\Big)+KE\Big( \mathbb{I}\{N_{(\kappa(n,s),s)}\geq2\}N^{p_1}_{(\kappa(n,s),s)}\Big)\notag
\\
&+Kn^{-p_1}+Kn^{-\frac{p_1}{2}+1}E\Big(\sum^m_{l=1}\sum_{k=0}^{N_{(\kappa(n,s),s)}}\int^{\tau_{k+1}}_{\tau_k}\Big(E\Big\{\big|\mathcal{D}b(X^n_{\tau_k},\alpha_{\tau_k})\big|^{2p_1}\Big|\mathscr{F}_T^{\alpha}\Big\}\notag
\\
&+E\Big\{\big|\tilde{\sigma}^{(l)}(r,X^n_{\kappa(n,r)},\alpha_{\kappa(n,r)})\big|^{2p_1}\Big|\mathscr{F}_T^{\alpha}\Big\}\Big)^{\frac{1}{2}}dr\Big)\notag
\\
\leq &Kn^{-p_1}+K P\big(N_{(\kappa(n,s),s)}=1\big)+K\sum_{N=2}^{\infty}N^{p_1} P\big(N_{(\kappa(n,s),s)}=N\big)\notag
\\
&+Kn^{-\frac{p_1}{2}+1}E\Big(\sum_{k=0}^{N_{(\kappa(n,s),s)}}\int^{\tau_{k+1}}_{\tau_k}\Big(E\Big\{\sup_{0\leq r\leq T}\big(1+|X^n_r|\big)^{2\rho p_1}\Big|\mathscr{F}_T^{\alpha}\Big\}\notag
\\
&+E\Big\{\sup_{0\leq r\leq T}\big(1+|X^n_r|\big)^{2 p_1}\Big|\mathscr{F}_T^{\alpha}\Big\}\Big)^{\frac{1}{2}}dr\Big)\notag
\\
\leq& Kn^{-p_1}+Kn^{-1}+K\sum^{\infty}_{N=2}N^{p_1}(qh)^N+Kn^{-\frac{p_1}{2}}\notag
\\
\leq& Kn^{-1}+Kh^2\sum^{\infty}_{N=0}(N+2)^{p_1}(qh)^N \notag
\end{align}
and by noticing that the infinite series in the above expression is convergence as $0<(1/n)<1/(2q)$, one obtains
\begin{align}
E|H_s+R_s|^{p_1} \leq Kn^{-1}\label{eq:H_R_Rate}
\end{align}
for any $p_1\geq 2$, $s\in[0,T]$ and $n\in \mathbb{N}$. Now, due to \eqref{eq:H_R_Rate}, H\"older's inequality, Remark \ref{rem:growth}, Assumption $H$-6 and Lemma \ref{lem:SchemeMoment}, one can estimate \eqref{eq:C_1} as,
\begin{align*}
C_1\leq & K\int^t_0 E\Big(\sup_{0\leq r\leq s}|e^n_r|^2\Big)ds+Kn^{-2}
\\
&+Kn^{-1}\int^t_0\Big\{\int^s_{\kappa(n,s)}E|b(X^n_r,\alpha_r)-b(X^n_{\kappa(n,r)},\alpha_r)|^2dr\Big\}^{\frac{1}{2}}ds
\\
&+Kn^{-1}\int^t_0\Big\{\int^s_{\kappa(n,s)}E|b(X^n_{\kappa(n,r)},\alpha_r)-b(X^n_{\kappa(n,r)},\alpha_{\kappa(n,r)})|^2dr\Big\}^{\frac{1}{2}}ds
\\
&+Kn^{-1}\int^t_0\Big\{\int^s_{\kappa(n,s)}E|b(X^n_{\kappa(n,r)},\alpha_{\kappa(n,r)})-b^n(X^n_{\kappa(n,r)},\alpha_{\kappa(n,r)})|^2dr\Big\}^{\frac{1}{2}}ds
\\
\leq &K\int^t_0 E\Big(\sup_{0\leq r \leq s}|e_r^n|^2\Big)ds+Kn^{-2}
\\
&+Kn^{-1}\int^t_0\Big\{\int^s_{\kappa(n,s)}E(1+|X^n_r|+|X^n_{\kappa(n,r)}|)^{2\rho}|X_r^n-X^n_{\kappa(n,r)}|^2dr\Big\}^{\frac{1}{2}}ds
\\
&+Kn^{-1}\int^t_0\Big\{\int^s_{\kappa(n,s)}E\Big(\mathbb{I}\{\alpha_r\neq \alpha_{\kappa(n,r)}\}E\Big\{(1+\sup_{0\leq r \leq T}|X^n_r|^{2\rho+2})\Big|\mathscr{F}_T^{\alpha}\Big\}\Big)dr\Big\}^{\frac{1}{2}}ds
\\
&+Kn^{-1}\int^t_0\Big\{\int^s_{\kappa(n,s)}E\Big(n^{-2}(1+\sup_{0\leq r \leq T}|X^n_r|^{2\rho_1+2})\Big)dr\Big\}^{\frac{1}{2}}ds
\\
\leq &K\int^t_0 E\sup_{0\leq r \leq s}|e_r^n|^2ds+Kn^{-2}\nonumber
\\
&+Kn^{-1}\int^t_0\Big\{\int^s_{\kappa(n,s)}\Big(E(1+|X^n_r|+|X^n_{\kappa(n,r)}|)^{4\rho}E\Big\{\sup_{0\leq r \leq T}|X_r^n-X^n_{\kappa(n,r)}|^4\Big\}\Big)^{\frac{1}{2}}dr\Big\}^{\frac{1}{2}}ds\nonumber
\\
&+Kn^{-1}\int^t_0\Big\{\int^s_{\kappa(n,s)} P(\alpha_r\neq \alpha_{\kappa(n,r)})dr\Big\}^{\frac{1}{2}}ds\nonumber
\end{align*}
and hence by Lemmas [\ref{lem:SchemeMoment}, \ref{lem:scheme_differ_rate}],  one obtains
\begin{align}
C_1\leq& K\int^t_0 E\sup_{0\leq r \leq s}|e_r^n|^2ds+Kn^{-2}\notag
\\
&+Kn^{-1}\int^t_0\Big\{\int^s_{\kappa(n,s)} \big(q_{\alpha_{\kappa(n,r)}\alpha_r}(r-\kappa(n,r))+o(r-\kappa(n,r))\big)dr\Big\}^{\frac{1}{2}}ds\notag
\\
\leq&K\int^t_0 E\sup_{0\leq r \leq s}|e_r^n|^2ds+Kn^{-2}\label{eq:C_1_Rate}
\end{align}
for any $t\in [0,T]$ and $n\in \mathbb{N}$. By using following splitting, 
\begin{align*}
\big(\sigma(X_r,\alpha_r)&-\tilde{\sigma}(r,X^n_{\kappa(n,r)},\alpha_{\kappa(n,r)})\big)=\big(\sigma(X_r,\alpha_r)-\sigma(X_r^n,\alpha_r)\big)
\\
&+\big(\sigma(X_r^n,\alpha_r)-\tilde{\sigma}(r,X^n_{\kappa(n,r)},\alpha_{\kappa(n,r)})\big),
\end{align*}
 $C_2$ can be written as
\begin{align*}
C_2:=&E\int^t_0\int^s_{\kappa(n,s)}\big(\sigma(X_r,\alpha_r)-\tilde{\sigma}(r,X^n_{\kappa(n,r)},\alpha_{\kappa(n,r)})\big)dW_r(H_s+R_s)ds
\\
=&E\int^t_0 \int^s_{\kappa(n,s)}(\sigma(X_r,\alpha_r)-\sigma(X_r^n,\alpha_r))dW_r (H_s+R_s)ds
\\
&+E\int^t_0 \int^s_{\kappa(n,s)}(\sigma(X_r^n,\alpha_r)-\tilde{\sigma}(r,X^n_{\kappa(n,r)},\alpha_{\kappa(n,r)}))dW_r (H_s+R_s)ds
\end{align*}
which by  Young's inequality, H\"older's inequality, Assumption $H$-2, \eqref{eq:H_R_Rate} and Lemma \ref{lem:sigmaCrusial} gives,
\begin{align}
C_2\leq& Kn\int^t_0E\Big|\int^s_{\kappa(n,s)}(\sigma(X_r,\alpha_r)-\sigma(X_r^n,\alpha_r))dW_r\Big|^2ds+Kn^{-1}\int^t_0E|H_s+R_s|^2ds\nonumber
\\
&+K\int^t_0\Big\{E\Big|\int^s_{\kappa(n,s)}(\sigma(X_r^n,\alpha_r)-\tilde{\sigma}(r,X^n_{\kappa(n,r)},\alpha_{\kappa(n,r)}))dW_r\Big|^2E|H_s+R_s|^2\Big\}^{\frac{1}{2}}ds\nonumber
\\
\leq& Kn\int^t_0\int^s_{\kappa(n,s)}|e^n_r|^2drds+Kn^{-2}\nonumber
\\
&+Kn^{-\frac{1}{2}}\int^t_0\Big\{\int^s_{\kappa(n,s)}E|\sigma(X_r^n,\alpha_r)-\tilde{\sigma}(r,X^n_{\kappa(n,r)},\alpha_{\kappa(n,r)})|^2dr\Big\}^{\frac{1}{2}}ds\nonumber
\\
\leq &K\int^t_0 E\sup_{0\leq r \leq s}|e_r^n|^2ds+Kn^{-2}\label{eq:C_2_Rate}
\end{align}
for any $t\in [0,T]$ and $n\in \mathbb{N}$. Also, notice that, 
\begin{align}
C_3:=&E\int^t_0e^n_{\kappa(n,s)}(H_s+R_s) ds=0\label{eq:C_3}
\end{align}
for any $t\in[0,T]$ and $n\in\mathbb{N}$. Moreover, by Young's inequality and H\"older's inequality one estimates
\begin{align*}
C_4+&C_5+C_6:=E\int^t_0 e^n_s\sum_{j_0\in \mathcal{S}}\int^s_{\kappa(n,s)}q_{\alpha_{r-}j_0}\Big(b(X^n_r,j_0)-b(X^n_r,\alpha_{r-})\Big)drds\nonumber
\\
&+E\int^t_0 e^n_s\sum_{k=0}^{N_{(\kappa(n,s),s)}}\Big(b(X^n_{\tau_{k+1}},\alpha_{\tau_k})-b(X^n_{\tau_k},\alpha_{\tau_k})-\mathcal{D}b(X^n_{\tau_k},\alpha_{\tau_k})(X^n_{\tau_{k+1}}-X^n_{\tau_k})\Big)ds\nonumber
\\
&+E\int^t_0 e^n_s\sum_{k=0}^{N_{(\kappa(n,s),s)}}\int^{\tau_{k+1}}_{\tau_k}\mathcal{D}b(X^n_{\tau_k},\alpha_{\tau_k})b^n(X^n_{\kappa(n,r)},\alpha_{\kappa(n,r)})drds
\\
\leq& K\int^t_0E\sup_{0\leq r \leq s}|e^n_r|^2ds+K\int^t_0E\Big|\sum_{j_0\in \mathcal{S}}\int^s_{\kappa(n,s)}q_{\alpha_{r-}j_0}\Big(b(X^n_r,j_0)-b(X^n_r,\alpha_{r-})\Big)dr\Big|^2ds
\\
&+K\int^t_0E\Big|\sum_{k=0}^{N_{(\kappa(n,s),s)}}\Big(b(X^n_{\tau_{k+1}},\alpha_{\tau_k})-b(X^n_{\tau_k},\alpha_{\tau_k})-\mathcal{D}b(X^n_{\tau_k},\alpha_{\tau_k})(X^n_{\tau_{k+1}}-X^n_{\tau_k})\Big)\Big|^2ds
\\
&+K\int^t_0E\Big|\sum_{k=0}^{N_{(\kappa(n,s),s)}}\int^{\tau_{k+1}}_{\tau_k}\mathcal{D}b(X^n_{\tau_k},\alpha_{\tau_k})b^n(X^n_{\kappa(n,r)},\alpha_{\kappa(n,r)})dr\Big|^2ds
\\
\leq &K\int^t_0E\sup_{0\leq r \leq s}|e^n_r|^2ds+K\int^t_0n^{-1}E\sum_{j_0\in \mathcal{S}}\int^s_{\kappa(n,s)}|b(X^n_r,j_0)-b(X^n_r,\alpha_{r-})|^2drds
\\
&+K\int^t_0E\Big((1+N_{(\kappa(n,s),s)})\sum_{k=0}^{N_{(\kappa(n,s),s)}}E\Big\{|b(X^n_{\tau_{k+1}},\alpha_{\tau_k})-b(X^n_{\tau_k},\alpha_{\tau_k})
\\
&\qquad\qquad-\mathcal{D}b(X^n_{\tau_k},\alpha_{\tau_k})(X^n_{\tau_{k+1}}-X^n_{\tau_k})|^2\Big|\mathscr{F}_T^{\alpha}\Big\}\Big)ds
\\
&+Kn^{-1}\int^t_0E\Big((1+N_{(\kappa(n,s),s)})\sum_{k=0}^{N_{(\kappa(n,s),s)}}\int^{\tau_{k+1}}_{\tau_k}E\Big\{|\mathcal{D}b(X^n_{\tau_k},\alpha_{\tau_k})
\\
&\qquad \qquad\times b^n(X^n_{\kappa(n,r)},\alpha_{\kappa(n,r)})|^2\Big|\mathscr{F}_T^{\alpha}\Big\}dr\Big)ds
\end{align*}
which by using Assumption $H$-5, H\"older's inequality, Remark \ref{rem:growth}, and Lemmas [\ref{lem:SchemeMoment}, \ref{lem: mvt}, \ref{lem:EN_i, EN_i^2 }, \ref{lem:scheme_differ_rate}] gives,
\begin{align}
C_4+&C_5+C_6\leq K\int^t_0E\sup_{0\leq r \leq s}|e^n_r|^2ds+Kn^{-1}\int^t_0\int^s_{\kappa(n,s)}E(1+|X^n_r|)^{2\rho+2}drds\nonumber
\\
&+K\int^t_0E\Big((1+N_{(\kappa(n,s),s)})\sum_{k=0}^{N_{(\kappa(n,s),s)}}E\Big\{(1+|X^n_{\tau_k}|+|X^n_{\tau_{k+1}}|)^{2\rho}|X^n_{\tau_{k+1}}-X^n_{\tau_k}|^4\Big|\mathscr{F}_T^{\alpha}\Big\}\Big)ds\nonumber
\\
&+Kn^{-1}\int^t_0E\Big((1+N_{(\kappa(n,s),s)})\sum_{k=0}^{N_{(\kappa(n,s),s)}}\int^{\tau_{k+1}}_{\tau_k}E\Big\{(1+\sup_{0\leq r \leq T}|X_r^n|^{4\rho+2})\Big|\mathscr{F}_T^{\alpha}\Big\}dr\Big)ds\nonumber
\\
\leq &K\int^t_0E\sup_{0\leq r \leq s}|e^n_r|^2ds+Kn^{-2}\nonumber
\\
&+K\int^t_0E\Big((1+N_{(\kappa(n,s),s)})\sum_{k=0}^{N_{(\kappa(n,s),s)}}\Big[E\Big\{(1+|X^n_{\tau_k}|+|X^n_{\tau_{k+1}}|)^{4\rho}\Big|\mathscr{F}_T^{\alpha}\Big\}\nonumber
\\
&\qquad \qquad\times E\Big\{|X^n_{\tau_{k+1}}-X^n_{\tau_k}|^8\Big|\mathscr{F}_T^{\alpha}\Big\}\Big]^{\frac{1}{2}}\Big)ds\nonumber
\\
&+Kn^{-1}\int^t_0E\Big((1+N_{(\kappa(n,s),s)})\sum_{k=0}^{N_{(\kappa(n,s),s)}}(\tau_{k+1}-\tau_k)\Big)ds\nonumber
\\
\leq & K\int^t_0E\sup_{0\leq r \leq s}|e^n_r|^2ds+Kn^{-2}\notag
\\
&+K\int^t_0E\Big((1+N_{(\kappa(n,s),s)})\sum_{k=0}^{N_{(\kappa(n,s),s)}}\big((\tau_{k+1}-\tau_k)^4+(\tau_{k+1}-\tau_k)^8\big)^{\frac{1}{2}}\Big)ds\nonumber
\\
\leq & K\int^t_0E\sup_{0\leq r \leq s}|e^n_r|^2ds+Kn^{-2}+Kn^{-2}\int^t_0E\big(1+N_{(\kappa(n,s),s)}\big)^2ds\notag
\\
\leq & K\int^t_0E\sup_{0\leq r \leq s}|e^n_r|^2ds+Kn^{-2}\label{eq:C_4,5,6Rate}
\end{align}
for any $t\in[0,T]$ and $n\in \mathbb{N}$.  Hence by combining the estimates from \eqref{eq:C_1_Rate} to \eqref{eq:C_4,5,6Rate} in \eqref{eq:C_Splitting}, one completes the proof.
\end{proof}
\begin{proof}[\textbf{Proof of Theorem \ref{thm:main}}]
Recall $e_t^n:=X_t-X^n_t$, $\textit{i.e.}$
\begin{align*}
e_t^n=e_0^n+\int^t_0\Big(b(X_s,\alpha_s)-b^n(X^n_{\kappa(n,s)},\alpha_{\kappa(n,s)})\Big)ds+\int^t_0\Big(\sigma(X_s,\alpha_s)-\tilde{\sigma}(s,X^n_{\kappa(n,s)},\alpha_{\kappa(n,s)})\Big)dW_s
\end{align*}
then use It\^o's formula for $|e^n_t|^2$ to get,
\begin{align*}
|e_t^n|^2=&|e^n_0|^2+2\int^t_0e^n_s\Big(b(X_s,\alpha_s)-b^n(X^n_{\kappa(n,s)},\alpha_{\kappa(n,s)})\Big)ds
\\
&+2\int^t_0e^n_s\Big(\sigma(X_s,\alpha_s)-\tilde{\sigma}(s,X^n_{\kappa(n,s)},\alpha_{\kappa(n,s)})\Big)dW_s
\\
&+\int^t_0 |\sigma(X_s,\alpha_s)-\tilde{\sigma}(s,X^n_{\kappa(n,s)},\alpha_{\kappa(n,s)})|^2ds
\end{align*}
almost surely for any $t\in[0,T]$ and $n\in\mathbb{N}$. By Assumptions $H$-1 and $H$-4, one obtains
\begin{align}
E|e^n_t|^2\leq & E|X_0-X_0^n|+2E\int^t_0 e^n_s\big(b(X_s,\alpha_s)-b^n(X^n_{\kappa(n,s)},\alpha_{\kappa(n,s)})\big)ds\nonumber
\\
&+E\int^t _0 |\sigma(X_s,\alpha_s)-\tilde{\sigma}(s,X^n_{\kappa(n,s)},\alpha_{\kappa(n,s)})|^2ds\nonumber
\\
=:&\,Kn^{-2}+ F_1+F_2 \label{eq:H_Splitting}
\end{align}
By using the splitting 
\begin{align*}
\big(b(X_s,\alpha_s)-b^n(X^n_{\kappa(n,s)},\alpha_{\kappa(n,s)})\big)&=\big(b(X_s,\alpha_s)-b(X^n_s,\alpha_s)\big)+\big(b(X^n_s,\alpha_s)-b(X^n_{\kappa(n,s)},\alpha_{\kappa(n,s)})\big)
\\
&+\big(b(X^n_{\kappa(n,s)},\alpha_{\kappa(n,s)}\big)-b^n\big(X^n_{\kappa(n,s)},\alpha_{\kappa(n,s)})\big),
\end{align*}
 $F_1$ can be written as,
\begin{align*}
F_1:=&2E\int^t_0 e^n_s\big(b(X_s,\alpha_s)-b^n(X^n_{\kappa(n,s)},\alpha_{\kappa(n,s)})\big)ds
\\
\leq &2E\int^t_0 e^n_s(b(X_s,\alpha_s)-b(X^n_s,\alpha_s))ds +2E\int^t_0 e^n_s(b(X^n_s,\alpha_s)-b(X^n_{\kappa(n,s)},\alpha_{\kappa(n,s)}))ds
\\
&+ 2E\int^t_0e^n_s(b(X^n_{\kappa(n,s)},\alpha_{\kappa(n,s)})-b^n(X^n_{\kappa(n,s)},\alpha_{\kappa(n,s)}))ds
\end{align*}
which due to Assumptions $H$-2, $H$-6, Young's inequality and Lemmas [\ref{lem:SchemeMoment}, \ref{lem:bCrusial}] yields
\begin{align}
F_1\leq & K\int^t_0 E|e^n_s|^2ds+Kn^{-2}+K\int^t_0 E|b(X^n_{\kappa(n,s)},\alpha_{\kappa(n,s)})-b^n(X^n_{\kappa(n,s)},\alpha_{\kappa(n,s)})|^2ds\nonumber
\\
\leq &K\int^t_0 E|e^n_s|^2ds+Kn^{-2}+Kn^{-2}\int^t_0E\sup_{0\leq s \leq T}(1+|X^n_s|)^{2\rho_1+2}ds \leq K\int^t_0 E|e^n_s|^2ds+Kn^{-2}\label{eq:H_1 Rate}
\end{align}
for any $t\in[0,T]$ and $n\in\mathbb{N}$. For $F_2$, one uses following splitting, 
\begin{align}
\sigma(X_s,\alpha_s)-\tilde{\sigma}(s,X^n_{\kappa(n,s)},\alpha_{\kappa(n,s)})&=(\sigma(X_s,\alpha_s)-\sigma(X^n_s,\alpha_s))\nonumber
\\
&+(\sigma(X^n_s,\alpha_s)-\tilde{\sigma}(s,X^n_{\kappa(n,s)},\alpha_{\kappa(n,s)})),\notag
\end{align}
 and then by the application of Assumption $H$-2 and Lemma \ref{lem:sigmaCrusial}, one obtains
\begin{align}
F_2:=&2E\int^t _0 |\sigma(X_s,\alpha_s)-\tilde{\sigma}(s,X^n_{\kappa(n,s)},\alpha_{\kappa(n,s)})|^2ds\nonumber
\\
\leq & KE\int^t_0|\sigma(X_s,\alpha_s)-\sigma(X^n_s,\alpha_s)|^2ds+KE\int^t_0|\sigma(X^n_s,\alpha_s)-\tilde{\sigma}(s,X^n_{\kappa(n,s)},\alpha_{\kappa(n,s)})|^2ds\nonumber
\\
\leq & K\int^t_0E|e^n_s|^2ds+Kn^{-2}\label{eq:H_3Rate}
\end{align}
for any $t\in[0,T]$ and $n\in\mathbb{N}$. By combining the estimates from \eqref{eq:H_1 Rate},  and \eqref{eq:H_3Rate} in \eqref{eq:H_Splitting}, yields
\begin{align*}
\sup_{0\leq t\leq u} E|e^n_t|^2\leq  K\int^u_0\sup_{0\leq r \leq s}E|e^n_r|^2ds+Kn^{-2}
\end{align*}
for any $u\in[0,T]$ and $n\in\mathbb{N}$. The Gronwall's lemma complete the proof.
\end{proof}

\begin{ack}
First author gratefully acknowledges financial support provided by Science and Engineering Research Board (SERB) under its MATRICS program through grant number SER-1329-MTD.
\end{ack}

\bibliographystyle{amsplain}

\begin{thebibliography}{20}

\bibitem{beyn2017}
W.-J. Beyn, E. Isaak, R.  Kruse (2017).  Stochastic C-stability and B-consistency of explicit and implicit Milstein-type schemes, \emph{J. Sci. Comput.},  70(3), 1042-1077. 

\bibitem{bao2016}
J. Bao and J. Shao (2016).  Permanence and extinction of regime-switching predator-prey models, \emph{SIAM J. Math. Anal.}, 48, 725-73. 

\bibitem{dareiotis2016} 
K. Dareiotis, C. Kumar, and S. Sabanis (2016).  On tamed Euler approximations of SDEs driven by L\'evy noise with applications to delay equations, \emph{SIAM J. Numer. Anal.},  54-3, 1840-1872.

\bibitem{davie2015} 
A. Davie (2015). Pathwise approximation of stochastic differential equations using coupling, \url{https://www.maths.ed.ac.uk/~sandy/coum.pdf}. 

\bibitem{giles2008} 
M. B. Giles  (2008).  Multilevel Monte Carlo path simulation, \emph{Oper. Res.},  56,  607-617.

\bibitem{kloeden1992}
P. E. Kloeden and E. Platen (1992). \emph{Numerical Solution of Stochastic Differential Equations}. Springer Berlin Heidelberg.

\bibitem{kumar2017}
C. Kumar and S. Sabanis (2017). On Explicit Approximations for L\'evy Driven SDEs with Super-linear Diffusion Coefficients, \emph{Electronic Journal of Probability}, 22, 1-19. 

\bibitem{kumar2019a}
C. Kumar and T. Kumar (2019). A Note on Explicit Milstein-Type Scheme for Stochastic Differential Equation with Markovian Switching, \emph{Preprint}. 

\bibitem{kumar2019}
C. Kumar, S. Sabanis (2019). On Milstein approximations with varying coefficients: the case of super-linear diffusion coefficients, to appear in \emph{BIT Numerical Mathematics}. 

\bibitem{mao2006}
X. Mao and C. Yuan (2006).  \emph{Stochastic Differential Equations with Markovian Switching}, Imperial College Press, London. 

\bibitem{mao2012}
S. L. Nguyen and G. Yin (2012). Pathwise convergence rates for numerical solutions of Markovian switching stochastic differential equations, \emph{Nonlinear Anal.: Real World Appl.},  13, 1170-1185. 

 \bibitem{nguyen2017}
S.L. Nguyen, T.A. Hoang, D.T. Nguyen and G. Yin (2017). Milstein-type procedures for numerical solutions of stochastic differential equations with Markovian switching, \emph{SIAM J. Numer. Anal.},  55 (2),  953-979.

\bibitem{nguyen2018}
D. T.Nguyen, S. L.Nguyen,  T. A. Hoang and G. Yin (2018). Tamed-Euler method for hybrid stochastic differential equations with Markovian switching, \emph{Nonlinear Analysis: Hybrid Systems}, 30, 14-30.

\bibitem{nguyen2012}
S. L. Nguyen and G.Yin (2012).  Pathwise convergence rates for numerical solutions of Markovian switching stochastic differential equations, \emph{Nonlinear Analysis: Real World Applications}, 13(3), 1170-1185.

\bibitem{sethi1994}
S. P. Sethi and Q. Zhang (1994).  \emph{Hierarchical Decision Making in Stochastic Manufacturing Systems}, Birkh\"auser, Boston.

\bibitem{tretyakov2013}
M. V. Tretyakov and Z. Zhang (2013). A fundamental mean-square convergence theorem for SDEs with locally Lipschitz coefficients and its applications, \emph{SIAM J. Numer. Anal.},  51(6), 3135-3162. 
 

\bibitem{yang2018} 
H. Yang and X. Li (2018). Explicit approximations for nonlinear switching diffusion systems in finite and infinite horizons, \emph{Journal of Differential Equations},  265(7), 2921-2967. 

\bibitem{yin2010}
G. Yin and C. Zhu (2010).  \emph{Hybrid Switching Diffusions, Properties and Applications}, Springer, New York.

\bibitem{zhang1998}
Q. Zhang (1998). Nonlinear filtering and control of a switching diffusion with small observation noise, \emph{SIAM J. Control Optim.}, 36, 1638-1668.

\bibitem{zhang2001}
Q. Zhang (2001). Stock trading: An optimal selling rule, \emph{SIAM J. Control Optim.}, 40, 64-87.


\end{thebibliography}

\end{document}